\documentclass[11pt,a4paper,english]{article}
\pdfoutput=1

\usepackage[affil-it]{authblk}
\usepackage{tikz}
\usetikzlibrary{decorations.pathreplacing}
\usepackage{amsthm, amsmath, amsfonts, mathtools,mathrsfs,enumerate}
\usepackage{color, verbatim}
\usepackage{todonotes}
\usepackage{graphicx}
\usepackage[hidelinks]{hyperref}

\newtheorem{thm}{Theorem}[section]
\newtheorem{cor}[thm]{Corollary}
\newtheorem{lem}[thm]{Lemma}

\theoremstyle{definition}
\newtheorem{defn}[thm]{Definition}
\newtheorem{assumption}[thm]{Assumption}
\theoremstyle{remark}
\newtheorem{rem}[thm]{Remark} 

\DeclareMathOperator*{\essinf}{ess\,inf}
\DeclareMathOperator*{\esssup}{ess\,sup}
\DeclareMathOperator*{\argmin}{arg\,min}
\DeclareMathOperator*{\argmax}{arg\,max}
\newcommand{\ie}{\textit{i.e.\ }}

\newcommand{\bbF}{\mathbb F}
\newcommand{\bbG}{\mathbb G}
\newcommand{\bbN}{\mathbb N}
\newcommand{\bbT}{\mathbb T}
\newcommand{\mcF}{\mathcal F}
\newcommand{\mcG}{\mathcal G}

\newcommand{\mcI}{\mathcal I}
\newcommand{\mcL}{\mathcal L}

\newcommand{\mcU}{\mathcal U}

\newcommand{\Prob}{\mathbb{P}}
\newcommand{\ett}{\mathbf{1}}

\newcommand{\R}{\mathbb R}
\newcommand{\EE}{\mathbb E}
\def\PP{\mathbb{P}}
\usepackage[ruled,linesnumbered]{algorithm2e}
\newcommand{\transp}{{\scriptscriptstyle\mathsf{T}}}

\newcommand{\msc}[1]{\textbf{MSC2010 Classification:} #1.}
\newcommand{\jel}[1]{\textbf{JEL Classification:} #1.}
\newcommand{\keywords}[1]{\textbf{Key words:} #1.}

\begin{document}

\title{Discrete-time risk-aware optimal switching with non-adapted costs}


\author[1]{Randall Martyr}
\author[1]{John Moriarty\thanks{Corresponding author. Email: j.moriarty@qmul.ac.uk}}
\author[2]{Magnus Perninge}

\affil[1]{School of Mathematical Sciences, Queen Mary University of London, Mile End Road, London E1 4NS, United Kingdom}
          
\affil[2]{Department of Physics and Electrical Engineering, Linnaeus University,
391 82 Kalmar, 351 95 V\"axj\"o, Sweden}

\date{\today}

\maketitle

\begin{abstract}
We solve non-Markovian optimal switching problems in discrete time on an infinite horizon, when the decision maker is risk aware and the filtration is general, and establish existence and uniqueness of solutions for the associated reflected backward stochastic difference equations. An example application to hydropower planning is provided.

\vspace{+4pt}
\vspace{+4pt}

\msc{60G40, 91B08, 49N30}
\vspace{+4pt}

\jel{C61, D81}
\end{abstract}

	\section{Introduction}

\subsection{Optimal switching problems}

Optimal switching problems involve an agent controlling a system by successively switching an operational mode between a discrete set of choices. Time may be either continuous or discrete, and in all cases the latter is useful for numerical work (see for example \cite{Ludkowski2010}). In related contexts, risk sensitivity with respect to uncertain costs has been modelled using nonlinear expectations, see \cite{An2013} for example. This feature is particularly appropriate in data-driven settings where models themselves may be uncertain. Examples include when the probability model is derived from numerical weather predictions depending on unknown physical parameters, or, alternatively, in model-free reinforcement learning. In the latter context, recent work has applied a general analytic framework for risk sensitivity \cite{Kose20}.

Taking a probabilistic approach, in this paper we consider a general filtration, which interacts with the nonlinear expectation. More precisely, let $\mathbb{T}$ be a subset of $\bbN_0 \coloneqq \{0,1,\ldots\}$ and $\{\tilde g_{\xi_{t-1},\xi_t}(t)\}_{t \in \mathbb{T}}$ be a sequence of random costs dependent on a switching strategy $\xi$, \ie a random sequence $(\xi_t)_{t \in \{-1\}\cup\bbT}$ taking values in a finite set $\mcI\coloneqq \{1,\ldots,m\}$, representing the set of operating modes. We do not require that every cost is observable which, for example, enables study of the interaction between delayed or missing observations and risk sensitivity. The time horizon is either infinite ($\bbT=\bbN_0$) or finite ($\bbT=\{0,1,\ldots,T\}$ for some finite $T\geq 0$) and the value of the switching problem is defined under a nonlinear expectation (cf. equations \eqref{ekv:optSwitch} and \eqref{eq:value-function-infinite-horizon} below).
Optimal stopping problems (see, for example, \cite{An2013}) are recovered in the special case of two modes (\ie $m=2$), when optimisation is performed over strategies $\xi$ with a single jump.

\subsection{Setup and related work}\label{sec:lr}

We have a probability space $(\Omega,\mathcal{F},\mathbb{P})$ and a filtration $\mathbb{G} = \{\mathcal{G}_{t}\}_{t \in \mathbb{T}}$ of sub-$\sigma$-algebras of $\mathcal{F}$. Given operating modes $\mathcal{I}\coloneqq\{1,\ldots,m\}$ and  essentially bounded random variables $g = \{g_{i}(t) \colon i \in \mathcal{I}\}_{t \in \mathbb{T}}$ and $c = \{c_{i,j}(t) \colon i,j \in \mathcal{I}\}_{t \in \mathbb{T}}$ on $(\Omega,\mathcal{F},\mathbb{P})$, we are interested in solving an optimal switching problem with running costs $g$ and switching costs $c$ when the information available to the decision maker is given progressively according to $\mathbb{G}$, and where a dynamic measure of risk sensitivity is used which generalises the usual sequence $\{\mathbb{E}[\cdot \vert \mathcal{G}_{t}]$, $t \in \mathbb{T}$\} of conditional expectations with respect to $\mathbb{G}$. For the following discussion we set
\begin{align}\label{eq:defcosts}
	\tilde{g}_{\xi_{t-1},\xi_t}(t) \coloneqq g_{\xi_t}(t) + c_{\xi_{t-1},\xi_t}(t).	
\end{align}
Note that we are considering a setting where each of the costs $\{\tilde g_{i,j}(t) \colon i,j \in\mcI\}_{t\in\mathbb T}$ is measurable with respect to the $\sigma$-algebra $\mcF$ and $\mathbb G$ is any filtration with $\mcG_t\subset \mcF$ for all $t\in\bbT$. We thus may have, but do not limit ourselves to, the situation where $\mathbb G$ is the natural filtration generated by $\{\tilde g_{i,j}(t) \colon i,j \in\mcI\}_{t\in\mathbb T}$. To our knowledge, the necessary and sufficient conditions we provide for an optimal switching strategy in this infinite-horizon setting under general filtration are novel and extend results in, for example, \cite{An2013,Cheridito2006,Follmer2016,Kratschmer2010,Pichler2018}.

The rest of the paper is structured as follows. Section \ref{sec:Non-Markovian-Optimal-Stopping} presents our main results in the finite-horizon setting, and these are extended to infinite horizon in Section~\ref{sec:infinite-horizon}. In both cases, the solution to the optimal switching problem is used to establish the existence of solutions to the associated reflected backward stochastic difference equations, and we also prove uniqueness of the solution. We close the paper with two examples. Section~\ref{sec:num-ex} briefly confirms that the approach taken to missing or delayed observations is capable of changing both the value process and optimal strategy. In Section~\ref{sec:num-ex-hydro-planning} we apply neural networks to obtain numerical solutions to a non-Markovian hydropower planning problem with non-adapted costs and examine the risk sensitivity of the solutions.

\section{Finite-horizon risk-aware optimal switching under general filtration}\label{sec:Non-Markovian-Optimal-Stopping}

In the following we let
\begin{itemize}
	\item $m\mcF$ denote the space of random variables on $(\Omega,\mathcal{F},\mathbb{P})$,
	\item $L^{\infty}_{\mathcal{F}}$ the subspace of essentially bounded random variables on $(\Omega,\mathcal{F},\mathbb{P})$,
	\item $\mathbb{G} = \{\mathcal{G}_{t}\}_{t \in \mathbb{T}}$ be a filtration, with $\mathcal{G} = \bigvee_{t \in \mathbb{T}}\mathcal{G}_{t}$ the $\sigma$-algebra generated by all $\mathcal{G}_{t}$ and $\mcG\subset\mcF$,
	\item $T<\infty$ be a finite time horizon and for $0 \le t \le T$, let $\mathscr{T}_{[t,T]}$ (resp. $\mathscr{T}_{t}$) denote the set of $\mathbb{G}$-stopping times with values in $t,\ldots,T$ (resp. $t,t+1,\ldots$), 	
	\item $\rho$ be a $\mathbb{G}$-conditional risk mapping:  a family of mappings $\{\rho_{t}\}_{t \in \mathbb{T}}$, $\rho_{t} \colon L^{\infty}_{\mathcal{F}} \to L^{\infty}_{\mathcal{G}_{t}}$, satisfying normalisation, conditional translation invariance, and monotonicity (see Appendix \ref{sec:conditional-risk-mapping}),
	\item for $s, t \in \mathbb{T}$ with $s \leq t$, let $\rho_{s,t}$ be the finite-horizon aggregated (or nested) risk mapping generated by $\rho$  (\cite{CHERIDITO2011,PichlerSchlotter2020,Pichler2018,ShenStannatObermayer2013,Ugurlu2018}, see also \cite{bauerle2018stochastic}): that is, $\rho_{t,t}(W_{t}) = \rho_{t}(W_{t})$ and 	
	\[
	\begin{split}
		\rho_{s,t}(W_{s},\ldots,W_{t}) =
		\rho_{s}\bigg(W_{s} & + \rho_{s+1}\Big(W_{s+1} + \cdots + \\
		& \qquad \qquad \rho_{t-1}\big(W_{t-1} + \rho_{t}(W_{t})\big)\cdots \Big)\bigg),\; s < t,
	\end{split}
	\]
	\item all inequalities be interpreted in the $\PP$-almost sure sense,
\end{itemize}
and for the finite time-horizon setting of Section \ref{sec:Non-Markovian-Optimal-Stopping} we also set $\mathbb{T} = \{0,1,\ldots,T\}$.

The value process for the finite-horizon optimal switching problem is
\begin{equation}\label{ekv:optSwitch}
	V_t^{i} \coloneqq \essinf_{\xi\in\mcU^i_t}\rho_{t,T}(\tilde g_{\xi_{t-1},\xi_t}(t),\ldots,\tilde g_{\xi_{T-1},\xi_{T}}(T)),
\end{equation}
where $\tilde g_{i,j}(t) \coloneqq g_j(t)+c_{i,j}(t)$ and $\mcU^i_t$ is the set of strategies $\xi$ with $\xi_{t-1} = i$ and the infimum of the empty set is taken to be $\infty$.
Since for each $t$ the costs $c_{i,i}(t)$ depend only on $i$ and may therefore be accounted for in the term $g_{i,i}(t)$, without loss of generality we may make the assumption
\begin{assumption}
	For all $i \in \mathcal{I}$ we have $c_{i,i}(t) = 0$ for all $t \in \mathbb{T}$.
\end{assumption}

\subsection{Dynamic programming equations}

The use of aggregated risk mappings provides sufficient structure for dynamic programming.
In our non-Markovian setting appropriate equations are:
\begin{align}\label{ekv:VFrec}
	\begin{cases}
		\hat V_T^{i}=\min_{j\in\mcI} \rho_T(\tilde g_{i,j}(T)),& {}\\
		\hat V_t^{i}=\min_{j\in\mcI} \rho_t(\tilde g_{i,j}(t)+\hat V^{j}_{t+1}),& \text{for} \: 0\leq t<T,
	\end{cases}
\end{align}
(the random fields $\{\hat V_t^{i}: i\in\mcI, t\in \mathbb{T}\}$ coincide with Snell envelopes, see Remark \ref{rem:osf}). We note by induction that $\hat V_t^{i} \in L^\infty_{\mcG_t}$ for each $i\in\mcI$ and $t \in \mathbb{T}$.

\begin{rem}
	For comparison, in a Markovian framework with full observation and the linear expectation, randomness stems from an $\mathbb{R}^{k}$-valued Markov chain $X^{s,x} \coloneqq \{X_t^{s,x}\}_{s \le t \le T}$, where $(s,x) \in \mathbb{T} \times \mathbb{R}^{k}$ is fixed and $X_r^{s,x}=x$ for $0 \le r \le s$ almost surely under $\mathbb{P}^{(s,x)}$, and $\mathbb{G}$ is the natural filtration of $X^{s,x}$.  In the Markovian case, by virtue of each strategy $\xi$ being adapted to $\mathbb{G}$, for every $t \ge 0$ there exists a function $\Xi_{t} \colon (\mathbb{R}^{k})^{t+1} \to \mathcal{I}$ such that $\xi_{t} = \Xi_{t}(X_{0},\ldots,X_{t})$. The Bellman equation is then the appropriate formulation for dynamic programming: For any $i \in \mcI$ and $(s,x) \in \mathbb{T} \times \mathbb{R}^{k}$,
	\begin{equation} \label{mainsyst-vi}
		\begin{cases}
			v^{i}(T,x) = \min_{j\in\mcI} \tilde{g}_{i,j}(T,x), \\
			v^{i}(s,x) = \min_{j\in\mcI } (\tilde{g}_{i,j}(s,x)+\EE^{(s,x)}[v^{j}(s+1,X_{s+1})]),\;\; s < T, \,
		\end{cases}
	\end{equation} where the $v^i$ and $\tilde{g}_{i,j}$ are deterministic functions on $\mathbb{T} \times \mathbb{R}^{k}$.
\end{rem}

\begin{thm}\label{thm:vrfctn}
	The random field $\{\hat V_t^{i}: i\in\mcI, t\in \mathbb{T}\}$ consists of value processes for the optimal switching problem, in the sense that
	\[
	\hat V_t^{i}=V_t^{i} \quad \forall\, i \in \mcI, t\in\bbT.
	\]
	Moreover, starting from any $0 \le t \le T$ and $i \in \mathcal{I}$, an optimal strategy $\xi^{*} \in \mcU^{i}_t$ can be defined as follows:
	\begin{equation}\label{eq:select}
		\begin{cases}
			\xi_{t-1}^* = i, \\
			\xi_s^* \in {\arg\min}_{j \in \mcI}\rho_s(\tilde g_{\xi_{s-1}^*,j}(s)+\hat V^{j}_{s+1}), \quad t \le s < T,\\
			\xi_T^* \in {\arg\min}_{j \in \mcI}\rho_T(\tilde g_{\xi_{T-1}^*,j}(T)).
		\end{cases}
	\end{equation}
\end{thm}

\begin{proof}
	Note that the result holds trivially for $t=T$. We will apply a backward induction argument and assume that for $s=t+1,t+2,\ldots,T$ and all $i\in\mcI$ we have $\hat V_s^{i}=V_s^{i}= \rho_{s,T}\left(\tilde g_{i,\xi_{s}}(s),\ldots,\tilde g_{\xi_{T-1},\xi_{T}}(T)\right)$, where $\xi_{t-1}=i$ and
	\[
	\xi_s\in {\arg\min}_{j\in\mcI}\rho_s(\tilde g_{\xi_{s-1},j}(s)+ V^{j}_{s+1}),
	\]
	with $V^{j}_{T+1} \coloneqq 0$ for all $j\in\mcI$.
	
	The induction hypothesis implies that
	\begin{align*}
		\hat V_t^{i}&= 	\min_{j\in\mcI} \rho_t(\tilde g_{i,j}(t)+ V^{j}_{t+1})
		\\
		&=\min_{j\in\mcI} \rho_t(\tilde g_{i,j}(t)+\essinf_{\xi\in\mcU^j_{t+1}}\rho_{t+1,T}(\tilde g_{j,\xi_{t+1}}(t+1),\ldots,\tilde g_{\xi_{T-1},\xi_{T}}(T))).
	\end{align*}
	For any $\xi'\in \mcU^i_{t}$ we note that by monotonicity and conditional translation invariance we have
	\begin{align*}
		\hat V_t^{i}&\leq\min_{j\in\mcI} \rho_t(\tilde g_{i,j}(t)+\rho_{t+1,T}(\tilde g_{j,\xi'_{t+1}}(t+1),\ldots,\tilde g_{\xi'_{T-1},\xi'_{T}}(T)))
		\\
		&\leq \sum_{j=1}^m \ett_{\{\xi'_t=j\}}\rho_t(\tilde g_{i,j}(t)+\rho_{t+1,T}(\tilde g_{j,\xi'_{t+1}}(t+1),\ldots,\tilde g_{\xi'_{T-1},\xi'_{T}}(T)))
		\\
		&=\rho_t\left(\sum_{j=1}^m\ett_{\{\xi'_t=j\}}\{\tilde g_{i,j}(t)+\rho_{t+1,T}(\tilde g_{j,\xi'_{t+1}}(t+1),\ldots,\tilde g_{\xi'_{T-1},\xi'_{T}}(T))\}\right)
		\\
		&=\rho_{t,T}(\tilde g_{i,\xi'_t}(t),\ldots,\tilde g_{\xi'_{T-1},\xi'_{T}}(T)).
	\end{align*}
	Taking the infimum over all $\xi'\in\mcU^i_{t}$ we conclude that $\hat V_t^{i}\leq V_t^{i}$. However, letting $\xi'_{t-1}=i$ and defining
	\[
	\xi'_s\in {\arg\min}_{j\in\mcI}\rho_t(\tilde g_{\xi'_{s-1},j}(s)+\hat V^{j}_{s+1}),
	\]
	for $s=t,\ldots,T$ with $\hat V^{j}_{T+1} \coloneqq 0$ for all $j\in\mcI$, we find that
	\begin{align*}
		\hat V_t^{i}&= \rho_t(\tilde g_{i,\xi'_t}(t)+\essinf_{\xi\in\mcU_{t+1}}\rho_{t+1,T}(\tilde g_{\xi'_t,\xi_{t+1}}(t+1),\ldots,\tilde g_{\xi_{T-1},\xi_{T}}(T)))
		\\
		&=\rho_{t,T}(\tilde g_{i,\xi'_t}(t),\ldots,\tilde g_{\xi'_{T-1},\xi'_{T}})
		\\
		&\geq V^i_t.
	\end{align*}
\end{proof}


\subsection{Relation to systems of RBS$\Delta$Es}

We now introduce a {\it reflected backward stochastic difference equation} (RBS$\Delta$E), which is a class of equations relevant to both optimal stopping and switching problems and studied systematically in \cite{An2013} for finite-state processes. Let $\mcL^{\infty}_{\bbG,T} \coloneqq \otimes_{t = 0}^{T}L^\infty_{\mcG_t}$ and, to avoid excessive notation, some notation for scalar-valued processes will be reused for vector-valued ones with the interpretation that all components are in the same space. Similarly, inequalities and martingale properties will be understood component-wise, and given $i \in \mathcal{I}$ we write $\mcI^{-i} \coloneqq \mcI \setminus\{i\}$.

\begin{defn}[Finite horizon RBS$\Delta$Es]\label{def:RBSDE}
	With $\bbT=\{0,\ldots,T\}$, where $0\leq T<\infty$, let $Y = \{Y_{t}\}_{t\in\bbT}$, $M = \{M_{t}\}_{t\in\bbT}$ and $A = \{A_{t}\}_{t\in\bbT}$ be $\mathbb{G}$-adapted $\R^m$-valued processes satisfying:
	\begin{equation}\label{ekv:rbsde}
		\begin{cases}
			Y^i_t = \min_{j\in\mcI}\rho_T(\tilde g_{i,j}(T))+\sum_{s=t}^{T-1}\rho_s(g_{i}(s)+\Delta M^i_{s+1})
			\\
			\qquad-(M^i_T-M^i_t) -(A^i_{T}-A^i_t),\quad\forall \;    t \in \mathbb{T},\\
			Y^i_t \leq  \min_{j\in\mcI^{-i}} \rho_t(\tilde g_{i,j}(t)+Y^j_{t+1}),\quad\forall \;    t \in \mathbb{T}, \\
			\sum_{t=0}^{T-1}\big(Y^i_t - \min_{j\in\mcI^{-i}} \rho_t(\tilde g_{i,j}(t)+Y^j_{t+1})\big)\Delta A^i_{t+1}=0.
		\end{cases}
	\end{equation}
	A triple $(Y,M,A)\in (\mcL^{\infty}_{\bbG,T})^3$ is said to be a solution to the system of RBS$\Delta$Es \eqref{ekv:rbsde} if M is a $\bbG$-adapted $\rho_{s,t}$-martingale (applying the definition in Section~\ref{sec:Aggregated-Martingales} of the appendix), $A$ is non-decreasing and $\bbG$-predictable (with $M_0=A_0=0$) and $(Y,M,A)$ satisfies \eqref{ekv:rbsde}. A solution $(Y,M,A)$ is called unique if any other solution $(Y',M',A')$ is indistinguishable as a process from $(Y,M,A)$.
\end{defn}

\begin{rem}
	The martingale characterisation of the optimal switching value process (see for example \cite{rieder1976optimal} under the linear expectation) may be derived from the associated RBS$\Delta$E. Under a risk mapping $\rho$, however, the ``driver'' $\rho_{t}\big(g_i(t) + \Delta M^i_{t+1}\big)$ in \eqref{ekv:rbsde} depends on the $\{\rho_{s,t}\}$-martingale difference $\Delta M^i_{t+1}$, which is natural for general (infinite state) backward stochastic difference equations -- see \cite{Cohen2011}. Note also that the driver is a function of the mappings $\omega \mapsto \Delta M^{i}_{t+1}(\omega)$ and $\omega \mapsto g_{i}(\omega,t)$ and not the realised values of these random variables. Also, we refer to the last line in equation \eqref{ekv:rbsde} as the {\it Skorokhod condition}.
\end{rem}

The optimal switching problem \eqref{ekv:optSwitch} is related to this system of reflected backward stochastic difference equations through the following result:

\begin{thm}\label{thm:rbsde}
	The system of RBS$\Delta$Es \eqref{ekv:rbsde} has a unique solution $(Y,M,A)$. Furthermore, we have $Y=V$.
\end{thm}

\begin{proof}
	We divide the proof into two parts:\\
	
	\emph{Existence:} We aim to find a family of $\rho_{s,t}$-martingales $M = \{M^{i}\}_{i \in \mathcal{I}}$ and non-decreasing $\bbG$-predictable processes $A = \{A^{i}\}_{i \in \mathcal{I}}$ such that $(V,M,A)$ solves \eqref{ekv:rbsde}. For every $i \in \mathcal{I}$ define the sequence $\{A^i_t\}_{t=0}^{T}$ by
	\[
	\begin{cases}
		A^i_0=0,\\
		A^i_{t}= A^i_{t-1}+\rho_{t-1}(g_i(t-1)+V^i_t)-V^i_{t-1}, \quad t = 1,\ldots,T.
	\end{cases}
	\]
	We note that $A^i$ is $\bbG$-predictable and non-decreasing since, by Theorem \ref{thm:vrfctn} and the backward induction formula \eqref{ekv:VFrec}, $V_t^{i}\leq \rho_t(g_{i}(t)+V^{i}_{t+1})$. Furthermore, for $t < T$ we have $\Delta A^i_{t+1}=\rho_{t}(g_i(t)+V^i_{t+1})-V^i_{t}=0$ on $\{V^i_{t}=\rho_{t}(g_i(t)+V^i_{t+1})\}\supset \{V^i_{t}<\min_{j\in\mcI^{-i}} \rho_t(\tilde g_{i,j}(t)+V^{j}_{t+1})\}$.
	
	Let $M^i$ be the martingale in the Doob decomposition (see Lemma \ref{Lemma:Doob-Decomposition} in Appendix \ref{sec:Aggregated-Martingales}) for $V^i$, that is, $M^i_0=0$ and $\Delta M^i_{t+1}=V^i_{t+1}-\rho_{t}(V^i_{t+1})$. We have
	\begin{align*}
		V^i_{t} &= \min_{j\in\mcI}\rho_T(\tilde g_{i,j}(T))+\sum_{s=t}^{T-1} (V^i_{s}-V^i_{s+1}).
	\end{align*}
	Now, as
	\begin{align*}
		\Delta M^i_{s+1}+\Delta A^i_{s+1}&=V^i_{s+1}-\rho_{s}(V^i_{s+1})+\rho_{s}(g_i(t)+V^i_{s+1})-V^i_{s}
		\\
		&=V^i_{s+1}+\rho_{s}(g_i(t)+V^i_{s+1}-\rho_{s}(V^i_{s+1}))-V^i_{s}
		\\
		&=V^i_{s+1}-V^i_{s}+\rho_{s}(g_i(t)+\Delta M^i_{s+1}),
	\end{align*}
	we get $V^i_{s}-V^i_{s+1}=\rho_{s}(g_i(s)+\Delta M^i_{s+1})-\Delta M^i_{s+1}+\Delta A^i_{s+1}$ and, thus,
	\begin{align*}
		V^i_{t} = {} & \min_{j\in\mcI}\rho_T(\tilde g_{i,j}(T))+\sum_{s=t}^{T-1}\rho_s(g_{i}(s)+\Delta M^i_{s+1}) -(M_T-M_t) \\
		& -(A_T-A_t).
	\end{align*}
	
	We conclude that $(V,M,A)$ is a solution to the RBSDE \eqref{ekv:rbsde}.\\
	
	\emph{Uniqueness:} Suppose that $(Y,N,B)$ is another solution. Then
	\begin{equation}\label{eq:uniqueness-proof-1}
		\Delta Y^i_{t+1} = -\rho_t(g_{i}(t)+\Delta N^i_{t+1})+\Delta N^i_{t+1}+\Delta B^i_{t+1}.
	\end{equation}
	Applying $\rho_t$ on both sides gives
	\begin{align*}
		\rho_t(\Delta Y^i_{t+1}) &=-\rho_t(g_{i}(t)+\Delta N^i_{t+1})+\rho_t(\Delta N^i_{t+1}+\Delta B^i_{t+1})
		\\
		&=-\rho_{t}(g_{i}(t)+\Delta N^i_{t+1})+\Delta B^i_{t+1}
	\end{align*}
	since, by our assumption on solutions to the RBSDE, $\Delta B^i_{t+1}$ is $\mcG_t$-measurable and $N^i$ is a martingale. Inserted into equation \eqref{eq:uniqueness-proof-1}, this gives
	\begin{align*}
		\Delta N^i_{t+1}&=\Delta Y^i_{t+1} + \rho_t(g_{i}(t)+\Delta N^i_{t+1})-\Delta B^i_{t+1}
		\\
		&=\Delta Y^i_{t+1}-\rho_t(\Delta Y^i_{t+1})
		\\
		&=Y^i_{t+1}-\rho_t(Y^i_{t+1})
	\end{align*}
	and
	\begin{align*}
		\Delta B^i_{t+1}&=\rho_t(\Delta Y^i_{t+1}) + \rho_t(g_{i}(t)+\Delta N^i_{t+1})
		\\
		&=\rho_t(\Delta Y^i_{t+1}) + \rho_t(g_{i}(t)+Y^i_{t+1}-\rho_t(Y^i_{t+1}))
		\\
		&=\rho_t(g_{i}(t)+Y^i_{t+1})-Y^i_t.
	\end{align*}
	We conclude that
	\begin{equation}\label{eq:One-Step-M-A-RBSDE}
		\begin{cases}
			\Delta N^{i}_{t+1}=Y^i_{t+1}-\rho_t(Y^i_{t+1})\\
			\Delta B^{i}_{t+1}=\rho_{t}(g_i(t)+Y^i_{t+1})-Y^i_t,
		\end{cases}
	\end{equation}
	and in particular we have that, given $Y\in\mcL^{\infty}_{\bbG,T}$, there is at most (up to indistinguishability of processes) one pair $(N,B)$ such that $(Y,N,B)$ solves the RBSDE \eqref{ekv:rbsde}.
	
	Since $(Y,N,B)$ solves the RBSDE \eqref{ekv:rbsde} we have that $$Y^i_t \leq  \min_{j\in\mcI^{-i}} \rho_t(\tilde g_{i,j}(t)+Y^j_{t+1}),$$ and
	\begin{align*}
		Y^i_t &= Y^i_{t+1}+\rho_t(g_{i}(t)+\Delta N^i_{t+1})-(N^i_{t+1}-N^i_t)-(B^i_{t+1}-B^i_t)
		\\
		&\leq Y^i_{t+1}+\rho_t(g_{i}(t)+\Delta N^i_{t+1})-(N^i_{t+1}-N^i_t)
		\\
		&= Y^i_{t+1}+\rho_t(g_{i}(t)+Y^i_{t+1}-\rho_t(Y^i_{t+1}))-(Y^i_{t+1}-\rho_t(Y^i_{t+1}))
		\\
		&=\rho_t(g_{i}(t)+Y^i_{t+1}).
	\end{align*}
	We conclude that $Y^i_t \leq  \min_{j\in\mcI} \rho_t(\tilde g_{i,j}(t)+Y^j_{t+1})$ for all $t \le T$ and $i \in \mcI$. For $t=T$ this implies that, for all $i \in \mcI$, $Y^i_T \leq  \min_{j\in\mcI} \rho_T(\tilde g_{i,j}(t))=V^i_T$. Assume that $t < T$ and $Y^i_{t+1} \leq  V^i_{t+1}$ for all $i \in \mcI$, then
	\begin{align*}
		Y^i_t &\leq  \min_{j\in\mcI} \rho_t(\tilde g_{i,j}(t)+Y^j_{t+1})
		\\
		&\leq \min_{j\in\mcI} \rho_t(\tilde g_{i,j}(t)+V^j_{t+1})
		\\
		&\leq V^i_t.
	\end{align*}
	Applying an induction argument we thus find that if $(Y,N,B)$ solves the RBSDE \eqref{ekv:rbsde} then $Y^i_t\leq V^i_t$ for all $t \le T$ and $i \in \mcI$. To arrive at uniqueness we show that the value $Y^i_t$ is attained by a strategy in which case the reverse inequality follows by optimality of $V^i_t$.
	
	Define the stopping time $\bar \tau^{t,i}_1 \coloneqq \inf \{ s\geq t:\Delta B^i_{s+1}>0\} \wedge T$ and the $\mcG_{\bar \tau^{t,i}_1}$-measurable $\mathcal{I}$-valued random variable $\bar \beta^{t,i}_1$ as a measurable selection of
	\[
	\begin{cases}
		\argmin\limits_{j\in\mcI^{-i}}\rho_{\bar \tau^{t,i}_1}\Big(\tilde g_{i,j}(\bar \tau^{t,i}_1)+Y^{j}_{\bar \tau^{t,i}_1+1}\Big), & \bar \tau^{t,i}_1 < T, \\
		\argmin\limits_{j\in\mcI}\rho_{T}\big(\tilde g_{i,j}(T)\big), & \bar \tau^{t,i}_1 = T.
	\end{cases}
	\]
	Now as $B^i_{\bar \tau^{t,i}_1}-B^i_t=0$ we have for $t\leq s< \bar \tau^{t,i}_1$ the recursion
	\begin{align*}
		Y^i_s &= Y^i_{s+1}+\rho_s(g_{i}(s)+\Delta N^i_{s+1})-(\Delta N^i_{s+1}) - (\Delta B^i_{s+1})
		\\
		&=Y^i_{s+1}+\rho_s(g_{i}(s)+Y^i_{s+1}-\rho_s(Y^i_{s+1}))-(Y^i_{s+1}-\rho_s(Y^i_{s+1}))
		\\
		&=\rho_s(g_{i}(s)+Y^i_{s+1}).
	\end{align*}
	Furthermore, by the Skorokhod condition, on $\{\bar \tau^{t,i}_1 < T\}$ we have that
	\begin{align*}
		Y^i_{\bar \tau^{t,i}_1}&=\min_{j\in\mcI^{-i}}\rho_{\bar \tau^{t,i}_1}(\tilde g_{i,j}(\bar \tau^{t,i}_1)+Y^{j}_{\bar \tau^{t,i}_1+1})
		\\
		&=\rho_{\bar \tau^{t,i}_1}(\tilde g_{i,\bar \beta^{t,i}_1}(\bar \tau^{t,i}_1)+Y^{\bar \beta^{t,i}_1}_{\bar \tau^{t,i}_1+1})
	\end{align*}
	and since $Y^{i}_{T} = \argmin\limits_{j\in\mcI}\rho_{T}\big(\tilde g_{i,j}(T)\big)$ we conclude that
	\begin{align*}
		Y^i_t=\rho_{t,\tau^{t,i}_1}(g_i(t),\ldots,g_i(\bar \tau^{t,i}_1-1),\tilde g_{i,\bar \beta^{t,i}_1}(\bar \tau^{t,i}_1)+Y^{\bar \beta^{t,i}_1}_{\bar \tau^{t,i}_1+1}),
	\end{align*}
	with $Y^{j}_{T+1} = 0$ for all $j \in \mathcal{I}$.
	
	This process can be repeated to define $\bar \tau^{t,i}_{k+1} \coloneqq \inf \{ s > \bar \tau^{t,i}_{k} \colon \Delta B^{\beta^{t,i}_k}_{s+1}>0\} \wedge T$ and the $\mcG_{\bar \tau^{t,i}_{k+1}}$-measurable $\mathcal{I}$-valued random variable $\bar \beta^{t,i}_{k+1}$ as a measurable selection of
	\[
	\begin{cases}
		\argmin\limits_{j\in\mcI^{-\beta^{t,i}_k}}\rho_{\bar \tau^{t,i}_{k+1}}\Big(\tilde g_{\beta^{t,i}_k,j}(\bar \tau^{t,i}_{k+1})+Y^{j}_{\bar \tau^{t,i}_{k+1}+1}\Big), & \tau^{t,i}_{k+1} < T,\\
		\argmin\limits_{j\in\mcI}\rho_{T}\big(\tilde g_{\beta^{t,i}_k,j}(T)\big), & \tau^{t,i}_{k+1} = T.
	\end{cases}
	\]
	Letting $\mathcal{N} \coloneqq \min\{k \ge 1 \colon \bar \tau^{t,i}_{k} \ge T\}$,
	\[
	\bar\xi^{t,i}_s \coloneqq i\mathbf{1}_{[-1, \tau^{t,i}_{1})}(s) + \sum_{j=1}^{\mathcal{N}} \beta_j \mathbf{1}_{[\bar \tau^{t,i}_{j},\bar \tau^{t,i}_{j+1})}(s) + \beta_{\mathcal{N}}\mathbf{1}_{\{s=T\}},
	\]
	and arguing as above we get that
	\begin{align*}
		Y^i_t&=\rho_{t,T}(\tilde g_{i,\bar\xi^{t,i}_t}(t),\ldots,\tilde g_{\bar\xi^{t,i}_{T-1},\bar\xi^{t,i}_{T}}(T))
		\\
		&\geq V^i_t.
	\end{align*}
\end{proof}

Given a strategy $\xi \in\mcU^{i}_t$, we can define its pairs of jump times $\tau_{j} \ge t$ and positions $\beta_{j} \in \mathcal{I}$ as follows:
\begin{equation}\label{eq:alternative-representation-of-strategy}
	\begin{split}
		\tau_{1} & = \inf \{s \ge t \colon \xi_{s} \neq i\} \wedge T, \\
		\beta_{1} & = \xi_{\tau_{1}}, \\
		\vdots \\
		\tau_{j+1} & = \inf \{s > \tau_{j} \colon \xi_{s} \neq \beta_{j}\} \wedge T, \\
		\beta_{j+1} & = \xi_{\tau_{j+1}}.
	\end{split}
\end{equation}
(Note that constant strategies $\xi_t \equiv i$ satisfy $\tau_{j} = T$ and $\beta_j = i$ for all $j$.)

We have the following characterisation of an optimal strategy:
\begin{cor}\label{cor:sufficient}
	A strategy $\xi \in\mcU^{i}_t$ is optimal for \eqref{ekv:optSwitch} if
	\begin{equation}\label{eq:optimal-strategy-sufficient-condition}
		\begin{cases}
			A^{\beta_{j-1}}_{\tau_j}-A^{\beta_{j-1}}_{\tau_{j-1}}=0,\\
			Y^{\beta_{j-1}}_{\tau_j}=\rho_{\tau_j}(\tilde g_{\beta_{j-1},\beta_{j}}+Y^{\beta_j}_{\tau_j+1}),
		\end{cases}
	\end{equation}
	where $\{(\tau_j,\beta_{j})\}$ are the pairs of jumps times and positions of $\xi$. If $\rho$ has the strong sensitivity property (cf. Appendix \ref{sec:conditional-risk-mapping}) then condition \eqref{eq:optimal-strategy-sufficient-condition} is also necessary for optimality.
\end{cor}
\begin{proof}
	\vskip 0.1em
	{\it Sufficiency:}
	
	From the proof of Theorem~\ref{thm:rbsde} we have that
	\begin{equation}\label{eq:value-of-optimal-strategy}
		Y^i_t = \rho_{t,T}(\tilde g_{i,\xi_t}(t),\ldots,\tilde g_{\xi_{T-1},\xi_{T}}),
	\end{equation}
	and optimality follows by the fact that $Y^i_t=V^i_t$.
	\vskip 0.1em
	{\it Necessity:} Suppose $\xi
	\in\mcU^{i}_t$ is optimal for \eqref{ekv:optSwitch} and $\rho$ is strongly sensitive. Let $\{(\tau_j,\beta_{j})\}$ be the pairs of jump times and positions of $\xi$ as defined in \eqref{eq:alternative-representation-of-strategy}. Then using \eqref{eq:value-of-optimal-strategy} above, Lemma \ref{lem:Recursive-Optional-Stopping}, the RBS$\Delta$Es \eqref{ekv:rbsde} and monotonicity of $\rho$ we have
	\begin{align*}
		Y^i_t & = \rho_{t,T}(\tilde g_{i,\xi_t}(t),\ldots,\tilde g_{\xi_{T-1},\xi_{T}}) \\
		& = \rho_{t,\tau_{1}}\Big(g_{i}(t),\ldots,g_{i}(\tau_{1}-1),\rho_{\tau_{1},T}\big(\tilde g_{i,\beta_{1}}(\tau_{1}),\ldots,\tilde g_{\xi_{T-1},\xi_{T}}\big)\Big) \\
		& \ge \rho_{t,\tau_{1}}\Big(g_{i}(t),\ldots,g_{i}(\tau_{1}-1),\rho_{\tau_{1}}\big(\tilde g_{i,\beta_{1}}(\tau_{1})+Y^{\beta_{1}}_{\tau_{1}+1}\big)\Big)\\
		& \ge \rho_{t,\tau_{1}}\big(g_{i}(t),\ldots,g_{i}(\tau_{1}-1),Y^{i}_{\tau_{1}}\big)\\
		& \vdots \\
		& \ge Y^{i}_{t},
	\end{align*}
	where we set $Y^{j}_{T+1} \coloneqq 0$ for all $j \in \mathcal{I}$. We therefore have
	\[
	\rho_{t,\tau_{1}}\Big(g_{i}(t),\ldots,g_{i}(\tau_{1}-1),\rho_{\tau_{1}}\big(\tilde g_{i,\beta_{1}}(\tau_{1})+Y^{\beta_{1}}_{\tau_{1}+1}\big)\Big) = \rho_{t,\tau_{1}}\big(g_{i}(t),\ldots,g_{i}(\tau_{1}-1),Y^{i}_{\tau_{1}}\big),
	\]
	and by strong sensitivity of $\rho$ and the definition of $A^{i}$ from Theorem~\ref{thm:rbsde}, \eqref{eq:optimal-strategy-sufficient-condition} is true for $j = 1$. The general case is proved by induction in a similar manner.
\end{proof}


\subsection{The special case of optimal stopping\label{sec:opti-stop}}
We now consider the problem of finding
\begin{align}\label{ekv:optStop}
	F_t\coloneqq\essinf_{\tau\in\mathscr{T}_{[t,T]}}\rho_{t,\tau}(f(t),\ldots,f(\tau-1),h(\tau)),
\end{align}
for given sequences $\{f(t)\}_{t=0}^{T}$ and $\{h(t)\}_{t=0}^{T}$ in $(L^{\infty}_{\mathcal{F}})^{T+1}$. This problem can be related to optimal switching with two modes $\mcI\coloneqq\{1,2\}$. The optimal stopping problem \eqref{ekv:optStop} is equivalent to \eqref{ekv:optSwitch} if we
\begin{itemize}
	\item Set $g_1(t)=f(t)$ for $0\leq t\leq T-1$, $g_1(T)=h(T)$, $c_{1,2}\equiv h$ and $g_2\equiv c_{2,1}\equiv 0$.
	\item \emph{Mutatis mutandis} let $\mcI$ depend on the present mode. We then set $\mcI(1) \coloneqq \{1,2\}$ when we are in mode $1$ and $\mcI(2) \coloneqq \{2\}$ whenever we are in mode 2. In particular this gives $\mcI(2)^{-2}=\emptyset$ in \eqref{ekv:rbsde}. We additionally use the conventions $\min\emptyset=\infty$ and $-\infty \cdot 0 = \infty \cdot 0 = 0$.
	\item Optimise over strategies satisfying $\xi_{t-1}=1$.
\end{itemize}
We note that in this setting the recursion \eqref{ekv:VFrec} gives $V^2\equiv 0$. The following result is then a direct consequence of Theorem \ref{thm:rbsde}:

\begin{thm}\label{Thm:Risk-Averse-Value-ch-Backward-Induction}
	The value process $F$ for the optimal stopping problem satisfies
	\begin{align}\label{ekv:VFrecOS}
		\begin{cases}
			F_T=\rho_T(h(T)),&{}\\
			F_t=\rho_t(f(t)+F_{t+1})\wedge\rho_t(h(t)),& \: 0\le t < T,
		\end{cases}
	\end{align}
	and the stopping time $\tau_{t} \in \mathscr{T}_{[t,T]}$ defined by
	\begin{equation}\label{eq:Risk-Averse-ch-Optimal-Stopping-Time}
		\tau_{t} = \inf\left\{ t \le s \le T \colon F_{s} = \rho_{s}(h(s))\right\},
	\end{equation}
	is optimal for \eqref{ekv:optStop}. Furthermore, there exist a $\rho_{s,t}$-martingale $M$ and a non-decreasing $\bbG$-predictable process $A$ such that $(F,M,A)$ is the unique solution to the following RBS$\Delta$E:
	
	\begin{equation}\label{ekv:rbsdeOS}
		\begin{cases}
			F_t = \rho_T(h(T))+\sum_{s=t}^{T-1}\rho_s(f(s)+\Delta M_{s+1})-(M_T-M_t)
			\\
			\qquad-(A_{T}-A_t),\quad \forall\,t\in\bbT,\\
			F_t \leq  \rho_t(h(t)), \quad \forall\,t\in\bbT,\\
			\sum_{t=0}^{T-1}(F_t - \rho_t(h(t)))\Delta A_{t+1}=0.
		\end{cases}
	\end{equation}
\end{thm}


\begin{rem}\label{rem:osf}
	As is done in \cite{Follmer2016}, Theorem \ref{Thm:Risk-Averse-Value-ch-Backward-Induction} can be used to identify the optimal switching problem with a family of optimal stopping problems. Setting
	\[
	h^i(t) = \begin{cases}
		\min_{j\in\mcI} \rho_T(\tilde g_{i,j}(T)),& t = T,\\
		\min_{j\in\mcI^{-i}} \rho_t(\tilde g_{i,j}(t)+\hat V^{j}_{t+1}),& t < T,
	\end{cases}
	\]
	and then substituting $h^i$ for $h$ in \eqref{ekv:VFrecOS} and recalling \eqref{ekv:VFrec}, it follows by Theorem 3 that for each $i \in \mcI$ and $t \in \mathbb{T}$ we have
	\[
	\hat V_t^{i} = \essinf_{\tau\in\mathscr{T}_{[t,T]}}\rho_{t,\tau}\big(\tilde g_{i,i}(t),\ldots,\tilde g_{i,i}(\tau-1),h^i(\tau)\big).
	\]
	
\end{rem}

\section{Infinite-horizon risk-aware optimal switching under general filtration}\label{sec:infinite-horizon}
In many problems the horizon $T$ is so long that it can be considered infinite, and this motivates us to extend the results obtained in Section~\ref{sec:Non-Markovian-Optimal-Stopping} to the infinite horizon. We thus let $\bbT \coloneqq \bbN_0$ and define the infinite-horizon aggregated risk mapping $\varrho_{s} \colon (L^{\infty}_{\mathcal{F}})^\mathbb{T} \to m\mcG_{s}$ (with $m\mcG_{s}$ the set of $\mcG_{s}$-measurable random variables) by
\begin{align}
	\varrho_s(W_s,W_{s+1},\ldots) = \limsup_{t \to \infty} \rho_{s,t}(W_s,W_{s+1},\ldots,W_t).
\end{align}

We define the value process for the switching problem on an infinite horizon as
\begin{equation}\label{eq:value-function-infinite-horizon}
	V_t^{i} \coloneqq \essinf_{\xi\in\mcU^i_t}\varrho_{t}(\tilde g_{\xi_{t-1},\xi_t}(t),\tilde g_{\xi_{t},\xi_{t+1}}(t+1),\ldots).
\end{equation}

\begin{defn}\label{def:Useful-Sequences-Random-Variables}
	\mbox{}
	Let $\mcL_\mathbb{G}^{\infty} \coloneqq \otimes_{t \in \mathbb{T}}L^\infty_{\mcG_t}$ and
	\[
	\mathcal{L}_\mathbb{G}^{\infty,d} \coloneqq \{W \in \mcL_\mathbb{G}^{\infty} \colon \lim_{s \to \infty} \esssup_\omega |W_s(\omega)| = 0 \}.
	\]
	Also, let $\mathcal{K}^{+}_{d}$ denote the set of all non-negative deterministic sequences $\{k_{t}\}_{t \in \mathbb{T}}$ such that the series $\sum_{t \in \mathbb{T}}k_{t}$ converges, and define
	\[
	H_\mathcal{F} \coloneqq {} \big\{ W \in (L_\mathcal{F}^\infty)^{\mathbb{T}} \colon \exists \{k_{t}\}_{t \in \mathbb{T}} \in \mathcal{K}^{+}_{d} \; \text{such that} \; |W_{t}| \le k_{t} \; \forall t \in \mathbb{T}\big\}.
	\]
\end{defn}
\begin{rem}\label{rem:ex}
	If $W \in H_\mathcal{F}$ then for every $s \in \mathbb{T}$ the limit
	\[
	\varrho_{s}(W_{s},W_{s+1},\ldots) = \lim_{t \to \infty}\rho_{s,t}(W_{s},\ldots,W_{t})
	\]
	exists almost surely and belongs to $L^{\infty}_{\mathcal{G}_{s}}$ (see Lemma \ref{lem:vrho-limit} in the appendix). An example $W \in H_\mathcal{F}$ is a discounted sequence $W_{t} = \alpha^t Z_t$ for some $\alpha \in (0,1)$ and $\{Z_{t}\}_{t \in \mathbb{T}} \subset L^\infty_\mathcal{F}$ with $\sup_{t}|Z_{t}| < C$ for some $C \in (0,\infty)$.
\end{rem}

\begin{assumption}\label{assumption:infinite-horizon}
	There exists a sequence $\{\bar g(t)\}_{t\in\bbT}\in H_{\mcF}$ such that $|\tilde g_{i,j}(t)| \leq \bar g(t)$ for all $(t,i,j)\in\bbT\times\mcI^2$.
\end{assumption}

\subsection{Dynamic programming equations}

For $(t,r) \in\bbT^{2}$ we set $\hat V_{t,r}^{i} \coloneqq \varrho_{t}(g_i(t),g_i(t+1),\ldots)$ whenever $t>r$ and define
\begin{align}\label{ekv:recIH}
	\hat V_{t,r}^{i}=\min_{j\in\mcI} \rho_t(\tilde g_{i,j}(t)+\hat V^{j}_{t+1,r}),
\end{align}
recursively for $t\leq r$. By a simple induction argument we note that for each $i\in\mcI$ and $r\in\bbT$ the sequence $\{\hat V_{t,r}^{i} \}_{t\in\bbT}$ exists as a member of $\mcL^{\infty,d}_\bbG$. We have the following lemma:\\

\begin{lem}\label{lem:IHtrunk}
	For $0\leq t\leq r$ and $i\in\mcI$ let $\mcU^i_{t,r}\coloneqq\{\xi\in\mcU^i_t:\xi_s=\xi_r,\,\forall s> r\}$. Then,
	\begin{align*}
		\hat V_{t,r}^{i}=\essinf_{\xi\in\mcU^i_{t,r}}\varrho_{t}(\tilde g_{\xi_{t-1},\xi_t}(t),\tilde g_{\xi_{t},\xi_{t+1}}(t+1),\ldots).
	\end{align*}
\end{lem}

\begin{proof} This follows immediately from Lemma~\ref{lem:vrho-rec} by applying Theorem~\ref{thm:vrfctn} with cost sequence
	\begin{align}
		\Big(\tilde g_{i,j}(t),\tilde g_{i,j}(t+1),\ldots,\tilde g_{i,j}(r-1),
		\tilde g_{i,j}(r)+\varrho_{r+1}\big(g_j(r+1),g_j(r+2),\ldots\big)\Big)_{i,j \in\mcI},
	\end{align}
	noting that $\varrho_{r+1}(g_j(r+1),g_j(r+2),\ldots)\in L^\infty_{\mcG_{r+1}}$.
\end{proof}

We arrive at the following verification theorem:

\begin{thm}\label{thm:infHOR}
	The pointwise limits $\{\tilde V_{t}^{i}\}_{t\in\bbT,i\in\mcI}\coloneqq\lim_{r\to\infty}\{\hat V_{t,r}^{i}\}_{t\in\bbT,i\in\mcI}$ exist and satisfy
	\begin{align*}
		\tilde V_{t}^{i}=V_t^{i},\quad\forall\,t\in\bbT.
	\end{align*}
	Furthermore, starting from any $t \in \mathbb{T}$ and $i \in \mathcal{I}$, the limit family defines an optimal strategy $\xi^* \in \mcU^i_t$ as follows:
	\[
	\begin{cases}
		\xi_r^* = i, \quad r < t,\\
		\xi_r^*\in {\arg\min}_{j \in\mcI}\rho_r\big(\tilde g_{\xi_{r-1}^*,j}(r)+\tilde V^{j}_{r+1}\big), \quad r \ge t.
	\end{cases}
	\]
\end{thm}

\begin{proof}
	From Lemma~\ref{lem:IHtrunk} and as $\mcU^i_{t,r}\subset \mcU^i_{t,r+1}\subset \mcU^i_{t}$ for all $0 \le t \le r$, the sequence $\{\hat V_{t,r}^{i}\}_{r\geq 0}$ is non-increasing and $\hat V_{t,r}^{i}\geq V_t^{i}$ for all $r\geq 0$. Further, as it is bounded from below by the sequence $\{\varrho_t(-\bar g(t),-\bar g(t+1),\ldots)\}_{t\in\bbT}$ (due to monotonicity) and $\{-\bar g(t)\}_{t \in \bbT} \in H_\mcF$, we conclude that the sequence $\big\{\{\hat V_{t,r}^{i}\}_{t\in\bbT,i\in\mcI}\big\}_{r\geq 0}$ converges pointwise.
	
	Now, by Assumption \ref{assumption:infinite-horizon} there is a non-negative decreasing deterministic sequence $\{K_s\}_{s\in \bbT}$, with $\lim_{s \to \infty}K_s = 0$, such that, for all $\xi\in\mcU^i_{t,r}$,
	\begin{align}\nonumber
		&|\varrho_{r+1}(\tilde g_{\xi_{r},\xi_{r+1}}(r+1),\tilde g_{\xi_{r+1},\xi_{r+2}}(r+2),\ldots)|
		\\
		& = \sum_{j\in\mcI}\ett_{\{\xi_r=j\}} |\varrho_{r+1}( g_{j}(r+1),g_{j}(r+2),\ldots)|\leq K_{r+1},\label{ekv:Kbnd}
	\end{align}
	and
	\begin{align}\label{ekv:Kbnd2}
		|\varrho_{r+1}(-\bar g(r+1),-\bar g(r+2),\ldots)|& \leq K_{r+1}.
	\end{align}
	Then, using Lemma \ref{lem:vrho-rec}, \eqref{ekv:Kbnd} gives
	\begin{align*}
		\hat V_{t,r}^{i} &= \essinf_{\xi\in\mcU^i_{t,r}}\varrho_{t}(\tilde g_{\xi_{t-1},\xi_t}(t),\tilde g_{\xi_{t},\xi_{t+1}}(t+1),\ldots)
		\\
		&\leq \essinf_{\xi\in\mcU^i_{t,r}} \rho_{t,r+1}(\tilde g_{\xi_{t-1},\xi_t}(t),\ldots,\tilde g_{\xi_{r-1},\xi_{r}}(r),K_{r+1})
		\\ &=\essinf_{\xi\in\mcU^i_{t,r}} \rho_{t,r}(\tilde g_{\xi_{t-1},\xi_t}(t),\ldots,\tilde g_{\xi_{r-1},\xi_{r}}(r))+K_{r+1},
	\end{align*}
	and \eqref{ekv:Kbnd2} implies that
	\begin{align*}
		V_t^{i}&=\essinf_{\xi\in\mcU^i_t}\varrho_{t}(\tilde g_{\xi_{t-1},\xi_t}(t),\tilde g_{\xi_{t},\xi_{t+1}}(t+1),\ldots)
		\\
		&\geq \essinf_{\xi\in\mcU^i_{t}}\rho_{t,r+1}(\tilde g_{\xi_{t-1},\xi_t}(t),\ldots,\tilde g_{\xi_{r-1},\xi_{r}}(r),-K_{r+1})
		\\
		&= \essinf_{\xi\in\mcU^i_{t,r}}\rho_{t,r}(\tilde g_{\xi_{t-1},\xi_t}(t),\ldots,\tilde g_{\xi_{r-1},\xi_{r}}(r))-K_{r+1},
	\end{align*}
	and we conclude that $\hat V_{t,r}^{i}-V_t^{i}\leq 2K_{r+1}$. Letting $r\to\infty$ gives  the first statement.\\
	
	For the second part, first note that the following inequality holds:
	\begin{equation}\label{eq:lower-infinite-horizon}
		V_t^{i} \ge \essinf_{\xi\in\mcU^i_{t,r}}\rho_{t,r}(\tilde g_{\xi_{t-1},\xi_t}(t),\ldots,\tilde g_{\xi_{r-1},\xi_{r}}(r) + V^{\xi_{r}}_{r+1}), \quad 0 \le t \le r.
	\end{equation}
	Indeed, for every $0 \le t \le r$ and $\xi \in \mcU^i_{t}$ we can use Lemma \ref{lem:vrho-rec} to argue that
	\begin{align*}
		\varrho_{t}\big(\tilde g_{\xi_{t-1},\xi_t}(t),\tilde g_{\xi_{t},\xi_{t+1}}(t+1),\ldots\big) & = \rho_{t,r}\big(\tilde g_{\xi_{t-1},\xi_t}(t),\ldots,\tilde g_{\xi_{r-1},\xi_{r}}(r) + \hat{V}_{r+1,r}^{\xi_{r}}\big) \\
		& \ge \rho_{t,r}\big(\tilde g_{\xi_{t-1},\xi_t}(t),\ldots,\tilde g_{\xi_{r-1},\xi_{r}}(r) + V^{\xi_{r}}_{r+1}\big) \\
		& \ge \essinf_{\xi'\in\mcU^i_{t,r}}\rho_{t,r}\big(\tilde g_{\xi'_{t-1},\xi'_t}(t),\ldots,\tilde g_{\xi'_{r-1},\xi'_{r}}(r) + V^{\xi'_{r}}_{r+1}\big),
	\end{align*}
	and since this is true for every $\xi \in \mcU^i_{t}$ we get \eqref{eq:lower-infinite-horizon}.
	Next, momentarily fix $0 \le t \le r$ and replace $g_{j}(r)$ with $g_{j}(r) + V^{j}_{r+1}$. Then using Theorem \ref{thm:vrfctn} with $T = r$ we have
	\begin{align*}
		V^i_t & \geq \essinf_{\xi\in\mcU^i_{t,r}}\rho_{t,r}(\tilde g_{\xi_{t-1},\xi_t}(t),\ldots,\tilde g_{\xi_{r-1},\xi_{r}}(r) + V^{\xi_{r}}_{r+1}) \\
		& = \rho_{t,r}\big(\tilde g_{i,\xi_{t}^*}(t),\tilde g_{\xi_{t}^*,\xi_{t+1}^*}(t+1),\ldots,\tilde g_{\xi_{r-1}^*,\xi_{r}^*}(r) + V^{\xi_{r}^*}_{r+1}\big)
		\\
		& \ge \rho_{t,r}\big(\tilde g_{i,\xi_{t}^*}(t),\tilde g_{\xi_{t}^*,\xi_{t+1}^*}(t+1),\ldots,\tilde g_{\xi_{r-1}^*,\xi_{r}^*}(r)\big) - K_{r+1}.
	\end{align*}
	Letting $r \to\infty $ we conclude that
	\begin{align*}
		V^i_t & \geq\varrho_{t}(\tilde g_{i,\xi_{t}^*}(t),\tilde g_{\xi_{t}^*,\xi_{t+1}^*}(t+1),\ldots),
	\end{align*}
	from which it follows that $\xi^*$ is an optimal strategy.
\end{proof}
We also record the following corollary which will be used in the proof of Theorem \ref{thm:rbsdeIH}.
\begin{cor}\label{cor:recIH}
	The value process for the infinite horizon optimal switching problem \eqref{eq:value-function-infinite-horizon} satisfies the following dynamic programming principle:
	\[
	V_t^{i} = \essinf_{\xi\in\mcU^i_{t,r}}\rho_{t,r}(\tilde g_{\xi_{t-1},\xi_t}(t),\ldots,\tilde g_{\xi_{r-1},\xi_{r}}(r) + V^{\xi_{r}}_{r+1}), \quad 0 \le t \le r.
	\]
\end{cor}

\begin{proof}
	We only need to prove that the following recursion holds:
	\begin{align}\label{eq:infrec}
		V^i_t=\min_{j\in\mcI} \rho_t(\tilde g_{i,j}(t)+V^{j}_{t+1}),
	\end{align}
	since the general result then follows from Theorem~\ref{thm:vrfctn} with $T = r$ and replacing $g_{j}(r)$ with $g_{j}(r) + V^{j}_{r+1}$. Taking limits on both sides in \eqref{ekv:recIH} gives
	\begin{align*}
		V^i_t&=\lim_{r\to\infty}\min_{j\in\mcI} \rho_t(\tilde g_{i,j}(t)+\hat V^{j}_{t+1,r})
		\\
		&\leq \lim_{r\to\infty}\min_{j\in\mcI} \rho_t(\tilde g_{i,j}(t)+ V^{j}_{t+1}+2K_{r+1})
		\\
		&=\lim_{r\to\infty}\{\min_{j\in\mcI} \rho_t(\tilde g_{i,j}(t)+ V^{j}_{t+1})+2K_{r+1}\}
		\\
		&=\min_{j\in\mcI} \rho_t(\tilde g_{i,j}(t)+ V^{j}_{t+1}).
	\end{align*}
	Since the reverse inequality follows as special case of \eqref{eq:lower-infinite-horizon}, the proof is complete.
\end{proof}

\subsection{Relation to systems of RBS$\Delta$Es}
\begin{defn}[Infinite horizon RBS$\Delta$Es]
	The infinite-horizon extension of Definition \ref{def:RBSDE} (with $\bbT=\mathbb N_0$) is given by:
	\begin{equation}\label{ekv:rbsdeIH}
		\begin{cases}
			Y^i_t = Y^{i}_{T} + \sum_{s=t}^{T-1}\rho_s(g_{i}(s)+\Delta M^i_{s+1})-(M^i_T-M^i_t)
			\\
			\qquad-(A^i_{T}-A^i_t),\quad \forall \;  t,T \in \mathbb{T}\; \text{with}\; t \le T,\\
			Y^i_t \leq  \min_{j\in\mcI^{-i}} \rho_t(\tilde{g}_{i,j}(t)+Y^j_{t+1}), \quad\forall \;   t \in \mathbb{T}, \\
			\sum_{t \in \mathbb{T} }\big(Y^i_t - \min_{j\in\mcI^{-i}} \rho_t(\tilde{g}_{i,j}(t)+Y^j_{t+1})\big)\Delta A^i_{t+1}=0.
		\end{cases}
	\end{equation}
	A solution to the system of RBS$\Delta$Es~\eqref{ekv:rbsdeIH} is a triple $(Y,M,A)\in \mcL^{\infty,d}_\bbG\times (\mcL_\mathbb{G}^{\infty})^2$ with $M$
	a $\rho_{s,t}$-martingale and $A$ a $\bbG$-predictable non-decreasing process.
\end{defn}

\begin{rem}
	In the special case when the limits $M_\infty=\lim_{t\to\infty}M_t$ and $A_\infty=\lim_{t\to\infty}A_t$ exist $\Prob$-a.s.~as members of $L^\infty_{\mcG}$, the infinite-horizon RBS$\Delta$E \eqref{ekv:rbsdeIH} can be written
	\begin{equation}\label{ekv:rbsdeIH2}
		\begin{cases}
			Y^i_t = \sum_{s=t}^{\infty}\rho_s(g_{i}(s)+\Delta M^i_{s+1})-(M^i_\infty-M^i_t) \\
			\qquad -(A^i_{\infty}-A^i_t),\quad\forall \;   t \in \mathbb{T},\\
			Y^i_t \leq  \min_{j\in\mcI^{-i}} \rho_t(\tilde{g}_{i,j}(t)+ Y^j_{t+1}), \quad\forall \;   t \in \mathbb{T},\\
			\sum_{t \in \mathbb{T}} (Y^i_t - \min_{j\in\mcI^{-i}} \rho_t(\tilde{g}_{i,j}(t)+ Y^j_{t+1}))\Delta A_{t+1}=0.
		\end{cases}
	\end{equation}
	We also emphasise that $Y \in \mcL^{\infty,d}_\bbG$ implies the boundary condition $\lim_{T \to \infty}Y^{i}_{T} = 0$ for all $i \in \mcI$.
\end{rem}

We have the following extension of Theorem \ref{thm:rbsde}:

\begin{thm}\label{thm:rbsdeIH}
	The system of RBS$\Delta$Es \eqref{ekv:rbsdeIH} has a unique solution. Furthermore, the solution satisfies $Y=V$.
\end{thm}

\begin{proof}
	\emph{Existence:} By Corollary~\ref{cor:recIH} the value process $V$ satisfies the following dynamic programming relation for any $T\in\bbT$:
	\begin{align*}
		V_t^{i}&=\essinf_{\xi\in\mcU^i_{t,T}}\rho_{t,T}(\tilde g_{\xi_{t-1},\xi_t}(t),\ldots,\tilde g_{\xi_{T-1},\xi_{T}}(T)+V^{\xi_{T}}_{T+1}), \quad 0 \le t \le T.
	\end{align*}
	Using Theorem~\ref{thm:rbsde}, this implies for every $T\in\bbT$ that $(V,M,A)$ is the unique solution to
	\[
	\begin{cases}
		V^i_t =  V^i_T+\sum_{s=t}^{T-1}\rho_s(g_{i}(s)+\Delta  M^i_{s+1})-( M^i_T- M^i_t)-( A^i_{T}- A^i_t),\\
		\qquad t=0,\ldots,T,\\
		V^i_t \leq  \min_{j\in\mcI^{-i}} \rho_t(\tilde g_{i,j}(t)+ V^j_{t+1}),\quad t=0,\ldots,T\\
		\sum_{t=0}^T( V^i_t - \min_{j\in\mcI^{-i}} \rho_t(\tilde g_{i,j}(t)+ V^j_{t+1}))\Delta  A^i_{t+1}=0,
	\end{cases}
	\]
	where $M^{i}_0 = A^{i}_0 = 0$ and
	\begin{align*}
		\left\{\begin{array}{l} \Delta M^{i}_{t+1}=V^i_{t+1}-\rho_t(V^i_{t+1}),
			\\
			\Delta A^{i}_{t+1}=\rho_{t}(g_i(t)+V^i_{t+1})-V^i_t.
		\end{array}\right.
	\end{align*}
	Furthermore, since this unique definition for the vector-valued processes $M$ and $A$ is independent of $T$, it follows that $(V,M,A)$ satisfies equation \eqref{ekv:rbsdeIH}.
	
	By the proof of Theorem \ref{thm:infHOR}, there exists a decreasing deterministic sequence $\{K_{t}\}_{t \in \mathbb{T}}$ such that $|V^i_T| \le K_{T}$ and $\lim_{T \to \infty}K_{T} = 0$. Therefore $V \in \mcL_\bbG^{\infty,d}$ and we conclude that $(V,M,A)$ is a solution to \eqref{ekv:rbsdeIH}.
	
	\vskip0.1em
	\emph{Uniqueness:} To show uniqueness, we note that if $(Y,N,B)$ is any other solution to~\eqref{ekv:rbsdeIH} then by again truncating at time $T \ge t$ and applying Theorem~\ref{thm:rbsde} we have that
	\begin{align*}
		Y_t^{i}&=\essinf_{\xi\in\mcU^i_t}\rho_{t,T}(\tilde g_{\xi_{t-1},\xi_t}(t),\ldots,\tilde g_{\xi_{T-1},\xi_{T}}(T)+Y^{\xi_{T}}_{T+1}).
	\end{align*}
	Since $Y, V \in\mcL_\bbG^{\infty,d}$ and $\mathcal{I}$ is finite, we can define a deterministic sequence $\{K_s\}_{s\in\bbT}$ with $K_s\to 0$ as $s\to\infty$ such that $|V^{i}_t-Y^{i}_t|\leq K_t$ for all $t\in\bbT$ and $i \in \mathcal{I}$. Appealing once more to the dynamic programming relation it follows that
	\begin{align*}
		Y_t^{i} & \leq \essinf_{\xi\in\mcU^i_t}\rho_{t,T}(\tilde g_{\xi_{t-1},\xi_t}(t),\ldots,\tilde g_{\xi_{T-1},\xi_{T}}(T)+V^{\xi_{T}}_{T+1}+K_{T+1}) \\
		& = V_t^{i} + K_{T+1},
	\end{align*}
	and similarly we have that $V_t^{i}\leq Y_t^{i}+K_{T+1}$. Letting $T\to\infty$ we find that $V_t^{i} = Y_t^{i}$ for all $i \in \mathcal{I}$ and uniqueness follows.
\end{proof}

\subsection{Relation to optimal stopping}
As an extension to Section~\ref{sec:opti-stop} above, we specialise to the case of optimal stopping on an infinite horizon:
\begin{equation}\label{ekv:optStopIH}
	F_t\coloneqq\essinf_{\tau\in\mathscr{T}_{t}}\rho_{t,\tau}(f(t),\ldots,f(\tau-1),h(\tau)).
\end{equation}
The above result for infinite-horizon optimal switching problems naturally extends the results in Section~\ref{sec:opti-stop} on optimal stopping in finite horizon to infinite horizon. We have the following:\\

\begin{cor}
	The value process $F$ satisfies the dynamic programming relation
	\begin{align*}
		F_{t}=\rho_t(f(t)+F_{t+1})\wedge\rho_t(h(t)),
	\end{align*}
	for all $t\in\bbT$ and an optimal stopping time $\tau^*_t$ is given by
	\begin{align*}
		\tau^*_t\coloneqq \inf\{s\geq t:F_s=\rho_s(h(s))\}.
	\end{align*}
	Furthermore, there exists a $\rho_{s,t}$-martingale $M$ and a non-decreasing $\bbG$-predictable process $A$ such that $(F,M,A)$ is the unique solution to the RBS$\Delta$E
	\begin{equation}\label{ekv:rbsdeinfHOROS}
		\begin{cases}
			F_t = F_{T} +\sum_{s=t}^{T-1}\rho_s(f(s)+\Delta M_{s+1})-(M_T-M_t)
			\\
			\qquad-(A_{T}-A_t),\quad \forall \;  t \in \mathbb{T} \;\text{and}\; T \in \mathbb{T}\; \text{with}\; t \le T,\\
			F_t \leq  \rho_t(h(t)), \quad\forall \; t \in \mathbb{T}, \\
			\sum_{t \in \mathbb{T}}(F_t - \rho_t(h(t)))\Delta A_{t+1}=0.
		\end{cases}
	\end{equation}
\end{cor}

\begin{proof} This follows immediately from Theorem~\ref{thm:infHOR} through the analogy between optimal switching problems and optimal stopping problems described in Section~\ref{sec:opti-stop}.\end{proof}


\section{Example: Delayed or missing observations}
\label{sec:num-ex}

In this section we aim to add some colour to the above results by illustrating the interplay between delayed or missing observations and risk awareness. We demonstrate that this issue should be treated differently than in the case of linear expectation, otherwise suboptimal actions may result.

Let $(\Omega, \mcF, \mathbb{F},\PP)$ be a filtered probability space and consider either the finite or infinite horizon problem above. Let the process of essentially bounded costs $(\tilde g_{i,j}(t), i,j \in \mathcal{I})_{t \in \mathbb{T}}$ be adapted (in the infinite horizon case, also satisfying Assumption \ref{assumption:infinite-horizon}) and let $\rho^\bbF$ be an $\mathbb{F}$-conditional risk mapping.
Suppose that the observation at some time $s$ is delayed. To model this, let ${\mathbb{G}}$ be the filtration given by
\[ {\mcG_t} =
\begin{cases}
	\mcF_{s-1}, & t = s, \\
	\mcF_{t}, & \text{otherwise,}
\end{cases}
\]
and let $\rho$ be the conditional risk mapping given by
\[ {\rho}_t =
\begin{cases}
	\rho^\bbF_{s-1}, & t = s, \\
	\rho^\bbF_{t}, & \text{otherwise.}
\end{cases}
\]
Indeed, since we will examine the decision taken at time $s$ rather than at later times, the observation at time $s$ may equivalently be missing rather than delayed.

For any time $t \in \mathbb{T}$ with $t \neq s$, the value processes at time $t$ are given by the dynamic programming equations \eqref{ekv:VFrec} or \eqref{eq:infrec} and conditional translation invariance:
\begin{align}\label{eq:vpt0}
	\hat V_t^{i}=\min_{j\in\mcI} \left(\tilde g_{i,j}(t)+ \rho_t(\hat V^{j}_{t+1})\right),
\end{align}
while the missing observation at time $s$ means that
\begin{align}\label{eq:vpt}
	\hat V_s^{i}&=\min_{j\in\mcI} \rho_s \left(\tilde g_{i,j}(s)+ \hat V^{j}_{s+1}\right), \\
	\xi_s^i &\in {\arg\min}_{j \in \mcI} \rho_s(\tilde g_{i,j}(s)+\hat V^{j}_{s+1}). \label{eq:goodxi}
\end{align}
When $\rho^\bbF$ is the linear (conditional) expectation, this is equivalent to
the following value and choice of mode:
\begin{align}
	\check V_s^{i}=\min_{j\in\mcI} \left({\rho}_s(\tilde g_{i,j}(s))+ {\rho}_s(\hat V^{j}_{s+1})\right), \\
	\check{\xi}_s^i \in {\arg\min}_{j \in \mcI}\left({\rho}_s(\tilde g_{i,j}(s)) + {\rho}_s(\hat V^{j}_{s+1})\right). 	\label{eq:badxi}
\end{align}
The intuitively obvious fact that the selections \eqref{eq:goodxi} and \eqref{eq:badxi} may differ can be confirmed by suitably modifying the costs at time $s$, as follows.
For $f \in m\mathcal{F}$ define
\begin{align}
	\begin{cases}
		\check C_{i,j}(f) &= {\rho}_s(\tilde g_{i,j}(s) )+ {\rho}_s(\hat V^{j}_{s+1}) - f, \\
		\hat C_{i,j}(f) &= {\rho}_s(\tilde g_{i,j}(s) + \hat V^{j}_{s+1}) - f.
	\end{cases}
	\label{eq:chatcheck}
\end{align}

We assume that
\begin{align}\label{eq:order}
	\check C_{i,j}(0) > \hat C_{i,j}(0) \text{ for each } i,j \in \mcI,
\end{align}
(which is true for example if the risk mapping $\rho^\bbF$ is subadditive),
and that for some $l \in \mcI$ we have
\begin{equation} \label{eq:assum1}
	\PP\left(\check C_{l,1}(0) - \hat C_{l,1}(0) = \check C_{l,2}(0) - \hat C_{l,2}(0) \; \right) < 1,
\end{equation}
setting $l = 1$ without loss of generality.
\begin{rem}
	Clearly, these assumptions fail when $\rho^\mathbb{F}$ is linear (and in the finite horizon case, they require that $s < T$). They can be understood as ensuring that $\rho^\mathbb{F}$ is `sufficiently nonlinear' on the problem data. The inequality \eqref{eq:order} serves to reduce combinatorial complexity.
\end{rem}
We argue as follows:
\begin{enumerate}
	\item Defining for each $n \in \mathbb{N}$ and $i=1,2$ the events
	\begin{align}\label{eq:an}
		A^i_n = \{\omega \in \Omega: \check C_{1,3-i}(0) - \hat C_{1,3-i}(0) > \check C_{1,i}(0) - \hat C_{1,i}(0) + 2/n\},
	\end{align}
	by assumption \eqref{eq:assum1} at least one of these events ($A^1_n$, say) has positive probability.
	\item Setting $f_{1,1} = \check C_{1,1}(0) - \check C_{1,2}(0) + 1/n$, we have
	\begin{align}\label{eq:choosef}
		\check C_{1,2}(0) = \check C_{1,1}(f_{1,1}) + 1/n.
	\end{align}
	\item We now further reduce combinatorial complexity by adjusting costs so that under both selections \eqref{eq:goodxi} and \eqref{eq:badxi}, when started in state $i=1$ at time $s-1$, at time $s$ only either remaining in state 1 or switching to mode 2 can be optimal. That is, we would like the following to hold:
	\begin{align}
		\begin{cases}
			&{\arg\min}_{j \in \mcI}\left\{\hat C_{1,j}(\bar f_{1,j})\right\} \subset \{1,2\}, \\
			&{\arg\min}_{j \in \mcI}\left\{\check C_{1,j}(\bar f_{1,j})\right\} \subset \{1,2\}.
		\end{cases} \label{eq:sel12}
	\end{align}
	By straightforward linear algebra and \eqref{eq:order}, this can be achieved by taking
	\begin{align}\label{eq:shift}
		\bar f = 1 + & \esssup \{\check C_{1,1}(f_{1,1}), \hat C_{1,1}(f_{1,1}), \check C_{1,2}(0),  \hat C_{1,2}(0)\} - \essinf_{k>2} \{\check C_{1,k}(0), \hat C_{1,k}(0) \} \nonumber\\
		= 1 + & \esssup \{\check C_{1,1}(f_{1,1}), \check C_{1,2}(0)\} -  \essinf_{k>2} \{\hat C_{1,k}(0)\}.
	\end{align}
	\item Finally, to observe a difference between the selections \eqref{eq:goodxi} and \eqref{eq:badxi}, set
	\begin{align}\label{eq:deff1j}
		\bar f_{1,j} =
		\begin{cases}
			\bar f+ f_{1,1}, & j=1, \\
			\bar f, & j=2, \\
			0, & j>2,
		\end{cases}
	\end{align}
	since then on $A^1_n$ we have
	\begin{align}\label{eq:bigorder}
		\check C_{1,2}(\bar f_{1,2}) > \check C_{1,1}(\bar f_{1,1}) > \hat C_{1,1}(\bar f_{1,1}) > \hat C_{1,2}(\bar f_{1,2}),
	\end{align}
	where the first inequality comes from combining \eqref{eq:chatcheck}, \eqref{eq:choosef} and \eqref{eq:deff1j}, the second from \eqref{eq:order}, and the third from combining \eqref{eq:an} with \eqref{eq:choosef}.
\end{enumerate}

Noting that $\bar f_{1,1}$ and  $\bar f_{1,2}$ are $\mcG_s$-measurable, it follows from \eqref{eq:chatcheck} that the two selections differ if we modify just the two costs $\tilde g_{1,1}(s)$ and $\tilde g_{1,2}(s)$ by replacing $\tilde g_{1,j}(s)$ with $\tilde g_{1,j}(s) - \bar f_{1,j}$ for $j \in \{1,2\}$.
In particular, from \eqref{eq:sel12} and \eqref{eq:bigorder}, if the system is in mode 1 at time $s-1$ then on $A^1_n$, \eqref{eq:badxi} selects mode 2 at time $s$ while \eqref{eq:goodxi} selects mode 1 at time $s$.

\section{Example: A hydropower planning problem\label{sec:num-ex-hydro-planning}}

In this section we first illustrate the above framework for risk-aware optimal switching under general filtration by formulating a non-Markovian hydropower planning problem (Sections \ref{sec:dsm}--\ref{sec:toptp}). In Sections \ref{sec:dpe}--\ref{sec:numerics} we provide practical dynamic programming equations, an approximate numerical scheme for the problem, a solution algorithm using neural networks and a discussion of numerical results.

\subsection{Decision space and market}\label{sec:dsm}
Consider a hydropower producer whose interventions take the form of bidding into a market. The producer sells electricity in a daily spot market at noon on the day before delivery. Let $T=9$, $\bbT:=\{0,\ldots,T\}$ and $\bbT^+:=\{0,\ldots,T+1\}$. Here, $t \in \bbT \cup \{-1\}$ represents a decision epoch at hour $12$ of day $t$, where day -1 is the last day of the previous planning period. We assume one-hour planning periods so at decision epoch $t\in \bbT$, the producer hands in a list of bids $B_{t}:=(B^E_{t+1,1},\ldots,B^E_{t+1,24};B^P_{t+1,1},\ldots,B^P_{t+1,24})$, where $B^E_{t+1,l}$ specifies the quantity of electrical energy offered and $B^P_{t+1,l}$ the acceptable price for hour $l$ of day $t+1$. Just after decision epoch $t$, the market clears and the prices of electricity are published. If the market price $R_{t+1,l}$ of electricity for hour $l$ exceeds the producer's bid price $B^P_{t+1,l}$, the producer is obligated to deliver the bidden volume $B^E_{t+1,l}$ of electrical energy during hour $l$ of day $t+1$. For this the producer receives a payment $R_{t+1,l}B^E_{t+1,l}$. The total income 
arising from the bid vector $B_t$ made at decision epoch $t$ is thus given by
\begin{align}\label{eq:rev1}
	\sum_{l=1}^{24}\ett_{\{B^P_{t+1,l}\leq R_{t+1,l}\}}R_{t+1,l}B^E_{t+1,l}.
\end{align}
If, on the other hand, a bid is accepted and the reservoir contains insufficient water to deliver the bidden volume, she has to purchase the undelivered energy from the balancing power market at a price $R^F$, which is usually higher than the spot price. This induces the cost
\begin{align}\label{ekv:fin-penal}
	\sum_{l=1}^{24}\ett_{\{B^P_{t+1,l}\leq R_{t+1,l}\}}R^F_{t+1,l}(B^E_{t+1,l}-E_{t+1,l})^+
\end{align}
of undelivered energy, where $E_{t,l}$ is the electrical energy produced during hour $l\in\{1,\ldots,24\}$ of day $t$.

\subsection{Probability space, inflow and price processes}

We take $(\Omega,\mcF,\{\mcF_t\}_{t \in \bbT},\Prob)$ to be a filtered probability space, with $\mcF_t$ representing the information available at noon on day $t \in \bbT$. This space will be rich enough to support a Markovian price process $(\tilde R_t)_{t\in\{-1\} \cup \bbT}$ and a non-Markovian inflow process $(\tilde I_t)_{t\in\bbT^+}$, as follows.

As is common in electricity planning problems, we assume that the electricity price vector $(\tilde R_t)_{t\in\bbT^+}:=(R_{t,1},\ldots,R_{t,24})_{t\in\bbT^+}$ is a bounded Markov process adapted to $\bbF$.
Regarding the inflow process, even under normal conditions, heavy rainfall only leads to increased inflows to a reservoir after a time delay, as the water is filtered through the catchment area surrounding the reservoir. Moreover, the hydropower station may be located in a mountainous area where river flows depend heavily on the melting of snow masses in a spring flood.
To model the discrete-time process of inflows $\{I_{t,j}\}^{1\leq j\leq 24}_{t\in\bbT}$, where $I_{t,l}$ is the inflow of water from the surroundings during hour $l$ of day $t$, let $(H_s)_{s\geq 0}$ be a continuous-time Markov process representing relevant environmental conditions. To account for the highly non-linear dependence of inflows on environmental conditions, set
\begin{align}\label{ekv:inflow}
	I_{t,j}=\int_{t+j/24}^{t+(j+1)/24}\int_{0}^{\delta}h(s)H_{r-s}dsdr,
\end{align}
where $\delta$ is a constant time lag and $h$ a deterministic function. Then $(\tilde I_t)_{t\in\bbT^+}:=(I_{t-1,13},\ldots,I_{t,12})_{t\in\bbT^+}$ is adapted to $\bbF$ and non-Markovian.

\subsection{Dynamics of the hydropower system}
We assume that the hydropower system consists of one reservoir containing the volume $M_{t,l}$ at the beginning of hour $l$ of day $t$ and a plant that produces electricity
\begin{align}\label{eq:elprod}
	E_{t,l} \coloneqq \eta(M_{t,l},F_{t,l}),
\end{align}
where $F_{t,l}$ is the flow of water directed through the turbines
and $\eta:\R^2_ +\to [0,C]$ is a deterministic function describing the efficiency of the plant with $C>0$ the installed capacity. We assume that the function $y \mapsto \eta(m,y)$ is strictly increasing for each fixed $m$ lying between the reservoir minimum level $M_\text{min}$ and maximum $M_\text{max}$.
The process $M = (M_{t,l})_{t,l}$ of reservoir levels follows the dynamics
\begin{align}\nonumber
	M_{t,l}&=\min\{\ett_{[l>1]}(M_{t,l-1}-F_{t,l-1}+I_{t,l-1})
	\\
	&\quad+\ett_{[l=0]}(M_{t-1,24}-F_{t-1,24}+I_{t-1,24}),M_{\text{max}}\},\label{ekv:reservoir-dynamics}
\end{align}
where $M_{0,13}$ is the volume in the reservoir at the first decision epoch.

Also, as explained in \cite{Lundstrom2020}, changing the production level by altering the flow $F_{t,l}$ may necessitate the startup or shutdown of turbines, resulting in both wear and tear and temporarily decreased efficiency. This feature motivates the inclusion of switching costs in the optimisation problem.

\subsection{The optimisation problem}\label{sec:toptp}
The controllable parameters in the problem are the bid vectors $\{B_t\}_{t\in\bbT}$. With the reasonable assumption that these bids take values in a finite set $\mcI\subset \R^{48}$ we have a switching problem. Let $\xi:=(\xi_t)_{t\in\bbT}$ denote the switching control, so that $\xi_t = B_t$ for each $t \in \bbT$.

By inverting $\eta$, the production plan and the reservoir level gives us the flow
\begin{align}\label{eq:flow}
	F_{t+1,l}=\min(f(B^E_{t+1,l},M_{t+1,l})\ett_{\{B^P_{t+1,l}\leq R_{t+1,l}\}}, M_{t+1,l} - M_{\text{min}}).
\end{align}
Substituting \eqref{eq:flow} into \eqref{ekv:reservoir-dynamics} we see that 
$M_{t}$ depends both on $\omega$ and on the entire history of $\xi$ up to time $t$. It follows that the switching costs are also dependent on this history. Therefore, recalling \eqref{eq:rev1}--\eqref{eq:flow} and letting $\mcI_t \coloneqq (\mcI)^{t+1}$, for $(i_{-1},\ldots,i_{t-1},i_{t}) \in  \mcI_{t+1}$
we may define the rewards for the planning problem as
\begin{equation}\label{eq:gij}
	\begin{split}
		& \tilde g_{i_{-1:t-1},i_{-1:t}}(t) \coloneqq {}  -c_{i_{-1:t}}(\tilde R_{t+1})+\sum_{l=1}^{24}\ett_{\{i_{t,24+l}\leq R_{t+1,l}\}}(R_{t+1,l}i_{t,l}
		\\
		& \qquad - R^F_{t+1,l}(i_{t,l}-\eta(M^{i_{-1:t}}_{t+1,l},
		\min(f(i_{t,l},M^{i_{-1:t}}_{t+1,l}),
		M^{i_{-1:t}}_{t+1,l} - M_\text{min})))^+) \\
		& \qquad + \ett_{\{t=T\}}R^M M^{i_{-1:T}}_{T+2,1},
	\end{split}
\end{equation}
where
\begin{itemize}
	\item 
	$i_{-1:t} := (i_{-1},\ldots,i_{t-1},i_{t})$;
	\item $M^{i_{-1:t}}_{t+1,l}$ is the reservoir level at hour $l$ on day $t+1$ corresponding to the bid history $i_{-1:t} \in \mcI_{t+1}$;
	\item $i_{t,m}$ is the $m$-th component of $i_t$;
	\item $R^M$ is the value of water stored at the end of the planning period;
	\item and for each
	$r\in\R^{24}_+$, $c_{i_{-1:t}}(r)$ is the cost rendered by switching from bid $i_{t-1}$ to $i_t$ when the price vector is $r$ and the bid history is $i_{-1:t}$.
\end{itemize}

If the producer has risk mapping $\rho$ then for each $t \in \mathbb{T}$, given a bid history $i_{-1:t-1} \in \mcI_{t}$
the objective is to find
\begin{equation}\label{ekv:value-fun-hydroplanning}
	V_t^{i_{-1:t-1}} \coloneqq \esssup_{\xi \in\mcU^{i_{-1:t-1}}_t}\rho_{t,T}\big(\tilde g_{\xi_{-1:t-1},\xi_{-1:t}}(t),\tilde g_{\xi_{-1:t},\xi_{-1:t+1}}(t+1),\ldots,\tilde g_{\xi_{-1:T-1},\xi_{-1:T}}(T)\big),
\end{equation}
where $\mcU^{i_{-1:t-1}}_t$ is the set of $\bbF$-adapted, $\mcI$-valued processes $(\xi_s)_{s \in\bbT}$ such that $\xi_{-1:t-1} = i_{-1:t-1}$. Note that the reward $\tilde g_{i_{-1:t-1},i_{-1:t}}(t)$ is $\mcF_{t+1}$-measurable but not $\mcF_{t}$-measurable. The producer's problem is thus one of non-adapted (in this case, delayed) information.

\subsection{Dynamic programming equations}\label{sec:dpe}
By modifying the proof of Theorem~\ref{thm:vrfctn} accordingly we can show that the value processes $(V_t^{i_{-1:t-1}} \colon i_{-1:t-1} \in \mcI_t)_{t\in\bbT}$ corresponding to \eqref{ekv:value-fun-hydroplanning} satisfy the following analogue of \eqref{ekv:VFrec}:
\begin{equation}\label{ekv:dynP-hydro}
	\begin{cases}
		V_T^{i_{-1:T-1}} = \max_{j\in \mcI} \rho_T(\tilde g_{i_{-1:T-1},(i_{-1:T-1},j)}(T)),& {}\\
		V_t^{i_{-1:t-1}} = \max_{j\in \mcI} \rho_t(\tilde g_{i_{-1:t-1},(i_{-1:t-1},j)}(t)+ V^{(i_{-1:t-1},j)}_{t+1}),& \text{for} \: 0\leq t<T,
	\end{cases}
\end{equation}
where for $i_{-1:t-1} \in \mcI_t$ and $j \in \mcI$ we define $(i_{-1:t-1},j) = (i_{-1},\ldots,i_{t-1},j)$. {\color{black} In order to obtain a practical solution algorithm we observe that the same optimal control can be obtained by dynamic programming without requiring the entire bid history. Recalling \eqref{eq:gij}, given $\omega \in \Omega$, for $(i_{-1},\ldots,i_{t-1},i_{t}) \in  \mcI_{t+1}$ the cost $\tilde g_{i_{-1:t-1},i_{-1:t}}(t)$ depends on $i_{-1:t-1}$ only through its final bid vector $i_{t-1}$ and the reservoir level $M^{i_{-1:t}}_{t+1,1}$. Moreover, by \eqref{ekv:reservoir-dynamics} and \eqref{eq:flow}, $M^{i_{-1:t}}_{t+1,1}$ only depends on $i_{-1:t-1}$ through $M_{t,13}^{i_{-1:t-1}}$ and the final bid vector $i_{t-1}$.
	Thus for $i_{-1:t}  \in \mcI_{t+1}$ and $m \in \mathbb{R}$ we may define new (random) rewards $\tilde{g}_{i_{t-1},i_{t}}(t, m)$ such that
	\begin{equation}\label{eq:gdef2}
		\begin{split}
			&\tilde g_{i_{t-1},i_{t}}(t,m) \coloneqq {}  -c_{i_{-1:t}}(R_{t+1})+\sum_{l=1}^{24}\ett_{\{i_{t,24+l}\leq R_{t+1,l}\}}(R_{t+1,l}i_{t,l}
			\\
			& \qquad - R^F_{t+1,l}(i_{t,l}-\eta(M^{m,i_{t}}_{t+1,l},\min(f(i_{t,l},M^{m,i_{t}}_{t+1,l}),M^{m,i_{t}}_{t+1,l}-M_\text{min})))^+) \\
			& \qquad + \ett_{\{t=T\}}R^M M^{i_{-1:T}}_{T+2,1},
		\end{split}
	\end{equation}
	where $M^{m,i_{t}}_{t+1}$ is the vector of reservoir levels on day $t+1$ given that on day $t$ the reservoir was at level $m$ at the beginning of hour $13$ (\textit{i.e.} at noon) and the bid vector was $i_{t}$. That is, $\tilde{g}_{i_{t-1},i_{t}}(t, m)$ and $\tilde g_{i_{-1:t-1},i_{-1:t}}(t)$ coincide when $M_{t,13}^{i_{-1:t-1}} = m$. Then define auxiliary value processes by
	\begin{equation}\label{redp2}
		V_t^{i}(m) =
		\begin{cases}
			\max_{j\in \mcI} \rho_T(\tilde g_{i,j}(T, m)),& {}\\
			\max_{j\in \mcI} \rho_t(\tilde g_{i,j}(t, m) + V_{t+1}^{j}(M^{m,j}_{t+1,13})) ,& \text{for} \: 0\leq t<T.
		\end{cases}
	\end{equation}
	
	By construction we have $V_t^{i_{t-1}}(M_{t,13}^{i_{-1:t-1}}) = V_t^{i_{-1:t-1}}$; this can be confirmed by backward induction. Therefore, if the auxiliary value function $V_{t+1}^{j}(m)$ can be computed for each $j \in \mcI$ and $m \in \mathbb{R}$, then \eqref{ekv:dynP-hydro} and \eqref{redp2} provide equivalent dynamic programming equations over the set of modes $\mcI$. The benefit of \eqref{redp2} is that we do not need to remember the switching control's entire history.
	Note that this reformulation is non-Markovian since $(M_t)_{t \in \mathbb{T}}$ is not a Markov process.
	In the next section we present a numerical approximation to this scheme using neural networks.}

\subsection{Numerical scheme}\label{sec:ns}
Let $\eta(M,F)=\eta_0MF$ with $\eta_0=0.1$ and $R_{t,l}=(1+|\sin(l\pi/12)|)(0\vee\tilde R_{t+(l-1)/24}\wedge C_R)$, where the multiplicative coefficient models the daily trend, $C_R=4$ is a price ceiling and $\tilde R$ solves the stochastic difference equation
\begin{align*}
	\tilde R_{t+1}-\tilde R_{t}=0.02(1-\tilde R_t)+0.05N_t,
\end{align*}
where $(N_t)_{t \in \mathbb{T}}$ are standard Normal random variables.

For the processes $I$ and $H$ of \eqref{ekv:inflow} we take $\delta=2/24$, $h(s):=\sin(s\pi/\delta)$ and $H$ to be a pure jump Markov process taking values in $\{0,0.5,1\}$
with transition intensity matrix
\begin{align*}
	Q_H:=\left[\begin{array}{rrr} -1 & 0.5 & 0.5\\
		1 & -2 & 1\\
		2 & 0.5 & -2.5\end{array}\right],
\end{align*}
representing no, medium and heavy rainfall respectively. For numerical purposes we approximate $H$ by a discrete-time Markov chain updating $k$ times per hour, with transition matrix $\exp\left(\frac{1}{24k}Q_H\right)$.

Moreover let $\tilde I$ be a discretisation of the set $[0,2]^{24}\times[0,4]^{24}$ (representing the fact that market bids  have limited precision, for example 1 MWh and 0.01 Euro) and let the hydropower producer's risk aversion be modelled by an entropic risk measure, \ie
\begin{align*}
	\rho_t(X)=-\frac{1}{\theta}\log\left(\EE\big[e^{-\theta X}\big|\mcF_t\big]\right),
\end{align*}
with parameter $\theta>0$. Finally, we assume that changes in production level cost $0.1$ Euro per MW and set $R^F=10$, $R^M:=4$, $M_\text{min} = 10$, $M_\text{max} = 50$ and $k=2$.

\subsubsection{State-space description}
To obtain a state-space description of our problem we introduce the state $(x_t)_{t\in\bbT}$, where $x_t$ is the non-redundant information available at hand at noon on day $t$, that is:
\begin{align}\label{eq:ssp}
	x_t:=\left[\begin{array}{c} M_{t,13}\\
		R_{t,24}\\
		\{P_{t,j}\}_{13\leq j\leq 24}\\
		\{H^k_{t+1/2-l/24k}\}_{l\in\{0,\ldots,24\delta k\}}\end{array}\right],
\end{align}
\noindent where $\{P_{t,j}\}_{13\leq j\leq 24}$ is the $\mcF_t$-measurable production schedule for the hours between noon and midnight of day $t$. In particular the state contains the discretised weather trajectory for the past two hours (10 am to noon) since, according to \eqref{ekv:inflow}, the impact of precipitation
is only fully revealed after this delay. Recalling the notation $\xi^*$ of Theorem~\ref{thm:vrfctn} for an optimal strategy, from  Section \ref{sec:dpe} the optimal mode (bid vector) $\xi_t^*$ depends its previous value $\xi_{t-1}^*$ only through the production schedule $ \{P_{t,j}\}_{13\leq j\leq 24}$, so we may write $\xi_t^* = \xi_t^*(x_t)$.

It follows from equations 
\eqref{ekv:reservoir-dynamics}--\eqref{eq:gij}  and \eqref{ekv:inflow} that given the system state $x_t$ and bid vector $j=B_t$ at noon on day $t$, both the reward $\tilde{g}_{i,j}(t)$ and the new state $x_{t+1}$ are measurable with respect to the noise vector $w_t$, where
\begin{align*}
	w_t:=\left[\begin{array}{c} \{R_{t+1,j}\}_{1 \leq j \leq 24}\\
		\{H^k_{t+1/2+l/24k}\}_{l\in\{1,\ldots,36k\}}\end{array}\right],
\end{align*}
which is not $\mcF_t$-measurable.

\subsection{Algorithm}\label{sec:alg}

In this section we describe an implementation of the numerical scheme of Section \ref{sec:ns}. Code implementing this scheme, and also a risk-neutral scheme, is available at \url{https://github.com/moriartyjm/optimalswitching/tree/main/hydro}
and is described in Algorithm \ref{alg}. For practicality it employs the neural networks shown in Figures \ref{fig:bidnn} and \ref{fig:valuenn}.

The bid neural network, whose architecture is given in Figure \ref{fig:bidnn}, aims to solve the following optimisation problem:
\begin{align}\label{ekv:VFrec2}
	\begin{cases}
		\xi^*_T(x_{T}) \in \argmax_{j\in\mcI} \left\{-\frac{1}{\theta}\log\left(E^{x_{T}}_{T}[e^{-\theta \tilde g_{i,j}(T)}] \right)\right\},& {}\\
		\xi^*_t(x_t) \in \argmax_{j\in\mcI} \left\{-\frac{1}{\theta}\log\left(E_{t}^{x_{t}}\big[e^{-\theta (\tilde g_{i,j}(t) + \hat{V}^{j}_{t+1}(x_{t+1}))}\big] \right)\right\},& \text{for} \: 0\leq t<T,
	\end{cases}
\end{align}
where $x_{t} \mapsto E_{t}^{x_{t}}$ approximates the conditional expectation with respect to $\mcF_t$ using the state vector, and  $\hat{V}^{j}_{t+1}(x_{t+1})$
approximates the continuation value using the value neural network, whose architecture is given in Figure \ref{fig:valuenn}. Continuation values $\hat{V}^{j}_{T+1}(x_{T+1})$ are set equal to zero. Note that these equations do not simplify further since the rewards $\tilde g_{ij}(t)$ are non-adapted.

The optimisation is performed by first training the bid neural network on $M$ independent noise realisations with target values equal to zero and loss function equal to $$-\frac{1}{\theta}\log\left(\frac{1}{M}\sum_{\ell=1}^{M}\Big[e^{-\theta (\tilde g_{i,\xi^*_t(x^{\ell}_{t})}(t) + \hat{V}^{j}_{t+1}(x^{\ell}_{t+1}))}\Big]\right),$$
where $x^{\ell}_{t}$ denotes the state vector $x_t$ under the $\ell$th noise realisation. (Note that since the state $x^{\ell}_{t}$ contains the production schedule $\{P^\ell_{t,j}\}_{13\leq j\leq 24}$, it also depends on the bid vector submitted at time $t-1$; we omit this dependency in order to lighten the notation). After the bid neural network has been trained, the value neural network
is trained on the $M$ independent noise realisations with target values equal to $\exp\big(-\theta (\tilde g_{i,\xi^*_t(x_t^\ell)}(t) + \hat{V}^{j}_{t+1}(x_{t+1}^\ell))\big)$ and the mean squared error as the loss function. Initial reservoir levels $M_{0,13}$ are drawn uniformly at random between $M_{\text{min}}$ and $M_{\text{max}}$, while initial market prices $R_{0,24}$ and weather values $\{H^k_{t+0.5-l/k}\}_{l\in\{0,\ldots,\delta k\}}=0$ are drawn from the corresponding stationary distribution.

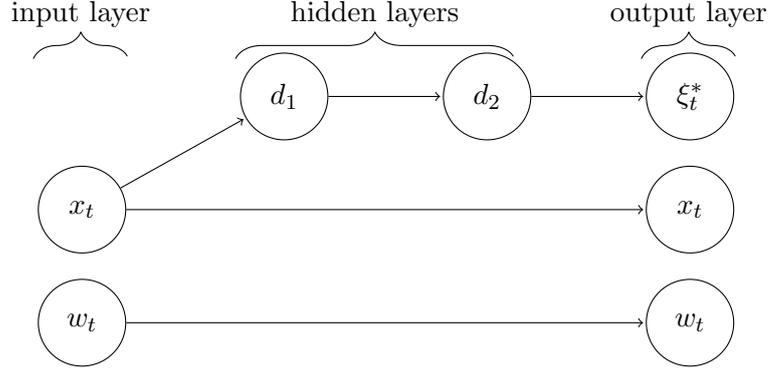
\begin{figure}
	\centering
	\begin{tikzpicture}[shorten >=1pt]
		\tikzstyle{unit}=[draw,shape=circle,minimum size=1.15cm]
		
		\node[unit](x1) at (0,3.5){$x_t$};
		\node[unit](x2) at (0,2){$w_t$};
		
		\node[unit](y1) at (8,5){$\xi^*_t$};
		\node[unit](y3) at (8,3.5){$x_t$};
		\node[unit](y4) at (8,2){$w_t$};
		
		\node[unit](w1) at (2.66,5){$d_1$};
		\node[unit](w2) at (5.33,5){$d_2$};
		
		\draw[->] (x1) -- (w1);
		\draw[->] (w1) -- (w2);
		
		\draw[->] (w2) -- (y1);
		
		\draw[->] (x1) -- (y3);
		\draw[->] (x2) -- (y4);

		\draw [decorate,decoration={brace,amplitude=10pt},xshift=-4pt,yshift=0pt] (-0.5,5.5) -- (0.75,5.5) node [black,midway,yshift=+0.6cm]{input layer};
		\draw [decorate,decoration={brace,amplitude=10pt},xshift=-4pt,yshift=0pt] (2.16,5.5) -- (5.83,5.5) node [black,midway,yshift=+0.6cm]{hidden layers};
		\draw [decorate,decoration={brace,amplitude=10pt},xshift=-4pt,yshift=0pt] (7.5,5.5) -- (8.75,5.5) node [black,midway,yshift=+0.6cm]{output layer};
	\end{tikzpicture}
	\caption{Architecture of the bid neural network. Nodes $d_1$ and $d_2$ represent dense layers with sigmoid activation function.}
	\label{fig:bidnn}
\end{figure}

\begin{figure}
	\centering
	\begin{tikzpicture}[shorten >=1pt]
		\tikzstyle{unit}=[draw,shape=circle,minimum size=1.15cm]
		
		\node[unit](x1) at (0,5){$x_t$};
		\node[unit](x2) at (0,3.5){$w_t$};
		
		\node[unit](y1) at (8,5){$\hat{V}_t$};
		
		\node[unit](w1) at (2.66,5){$d_3$};
		\node[unit](w2) at (5.33,5){$d_4$};
		
		\draw[->] (x1) -- (w1);
		\draw[->] (x2) -- (w1);
		\draw[->] (w1) -- (w2);
		
		\draw[->] (w2) -- (y1);
		
		\draw [decorate,decoration={brace,amplitude=10pt},xshift=-4pt,yshift=0pt] (-0.5,5.5) -- (0.75,5.5) node [black,midway,yshift=+0.6cm]{input layer};
		\draw [decorate,decoration={brace,amplitude=10pt},xshift=-4pt,yshift=0pt] (2.16,5.5) -- (5.83,5.5) node [black,midway,yshift=+0.6cm]{hidden layers};
		\draw [decorate,decoration={brace,amplitude=10pt},xshift=-4pt,yshift=0pt] (7.5,5.5) -- (8.75,5.5) node [black,midway,yshift=+0.6cm]{output layer};
	\end{tikzpicture}
	\caption{Architecture of the value neural network. To reduce dimension, in the state vector $x_t$ the production schedule $\{P_{t,j}\}_{13\leq j\leq 24}$ is replaced by the sum of its entries. Nodes $d_3$ and $d_4$ represent dense layers with sigmoid activation function.}
	\label{fig:valuenn}
\end{figure}
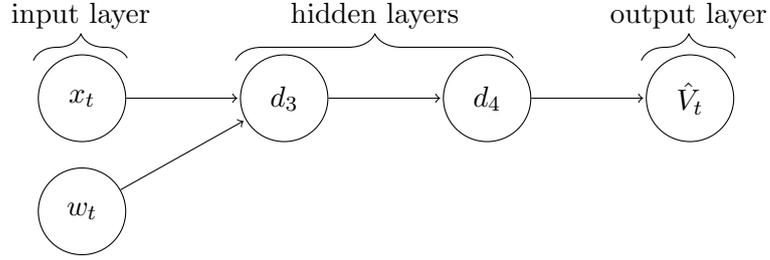

\begin{algorithm}[!h]\label{alg}
	\SetKwInOut{Input}{Input}
	\SetKwInOut{Output}{Output}
	\SetKw{In}{in}
	\SetKw{State}{State:}
	\SetKw{Noise}{Noise:}
	\SetKw{Train}{Train}
	\SetKw{Store}{Store}
	\SetKw{Predict}{Predict}
	\Input{$M$ independently sampled price and weather trajectories and initial reservoir levels}
	\For{$t$ \In $\{T,T-1,\ldots,0\}$}
	{
		Train bid NN for day $t$ with states $(x_t^i)_{i=1,\ldots,M}$ and noise $(w_t^i)_{i=1,\ldots,M}$, store as {\tt model\_bids[$t$]}, and predict optimal bid vectors $((\xi^*_t)^i)_{i=1,\ldots,M}$ \\
		Train value NN for day $t$ with states $(x_t^i)_{i=1,\ldots,M}$, bid vectors $((\xi^*_t)^i)_{i=1,\ldots,M}$ and noise $(w_t^i)_{i=1,\ldots,M}$ and
		store as {\tt mdl\_E\_exp[$t$]}
	}
	\Output{Neural networks {\tt model\_bids}, {\tt mdl\_E\_exp} approximating optimal bid vector and value for each day $t$, state $x_t$
	}
	\caption{Hydropower planning over $T$ days}
\end{algorithm}

\subsection{Numerical results and discussion}\label{sec:numerics}
In this section we present and discuss numerical results obtained using Algorithm~\ref{alg} over an optimisation horizon of 10 days and with 50,000 independent noise realisations. Identifying the risk-neutral case with $\theta = 0$, results are plotted for $\theta=0, 0.01$ and $0.02$ in blue (solid), orange (dashed) and green (dotted) respectively.

For each hour in the optimisation, Figure~\ref{fig:meanGEN} shows the production level under the respective optimal strategies, averaged across all noise realisations. Similarly, Figure~\ref{fig:M_t-samples} plots the mean water level under the optimal strategies, together with the 0.05 percentiles (dashed lines).
\begin{figure}[h!]
	\centering
	\includegraphics[width=0.75\textwidth]{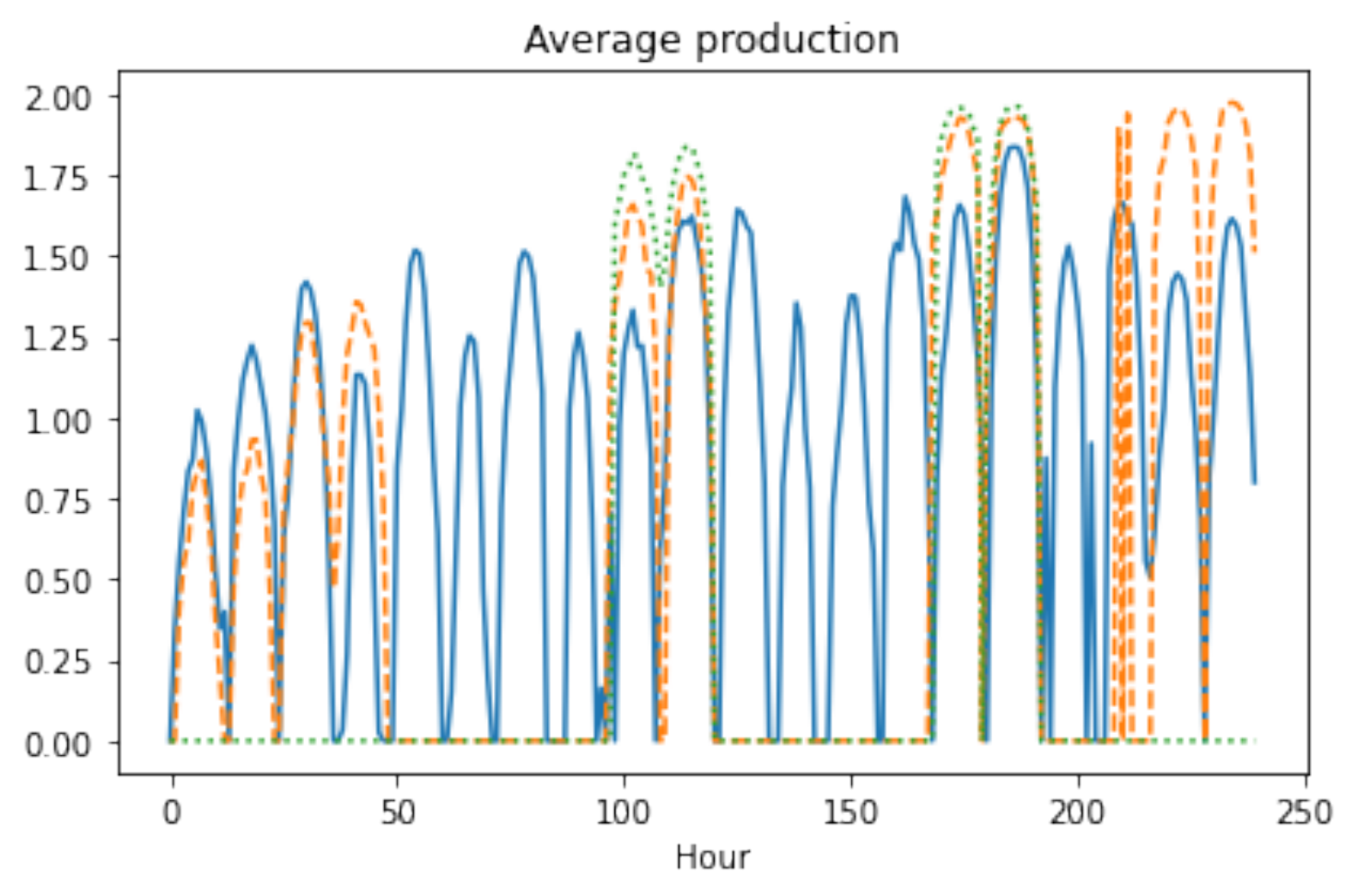}
	\caption{Average production curves by hour for risk sensitivity $\theta=0, 0.01$ and $0.02$ (blue solid, orange dashed and green dotted respectively).
	}
	\label{fig:meanGEN}
\end{figure}
In order to represent the value processes, Figure~\ref{fig:VFs} plots the prediction $\hat{V}_0(x_0)$ made by the value neural network when the input is $x_0=\left[m, 0, \mathbf{0}, \mathbf{0}\right]^{\transp}$, for different values of $m$ (recall \eqref{eq:ssp}, ${}^\transp$ denotes transpose).

The reservoir's physical constraints $M_{\text{min}}$ and $M_{\text{max}}$ create risks for the hydropower producer. When the reservoir level is near $M_{\text{min}}$ the producer risks being unable to fulfil the bid volume and receiving a penalty for under-production. Conversely, if the reservoir reaches its maximum level $M_{\text{max}}$ then she risks spilling the water inflow, which would otherwise be stored and used profitably later.

From Figure~\ref{fig:M_t-samples}, the risk-neutral producer maintains the reservoir at an intermediate water level on average. Further, in at least 5\% of cases she allows the water level to fall rather close to the minimum level. In contrast, in at least 95\% of cases the optimal strategy of the risk-averse producers first drives the initial water level up by trading less, and production increases only once the reservoir is at least approximately half filled. Indeed, for $\theta=0.02$ the average water level is seen to increase towards $M_{\text{max}}$ over the time horizon. Thus increases in $\theta$ incentivise the producer to avoid the risk of under-production penalties. (The risk of spilling water at level  $M_{\text{max}}$ appears to have comparatively less influence on the optimal strategies.)	

These observations are also borne out in Figure~\ref{fig:VFs}. In the risk-neutral case, the marginal value of water is approximately constant as the water level varies. However locally around $M_{\text{min}}$, where the risk of penalties has more influence, the marginal value of water becomes lower as the risk sensitivity parameter $\theta$ increases.

Figure~\ref{fig:meanGEN} confirms that the risk-neutral strategy involves producing every day, and also involves following the daily price trend within each day. As the risk aversion parameter $\theta$ increases, the number of production days, and also the total produced volume, decrease.

\begin{figure}[h!]
	\centering
	\includegraphics[width=0.75\textwidth]{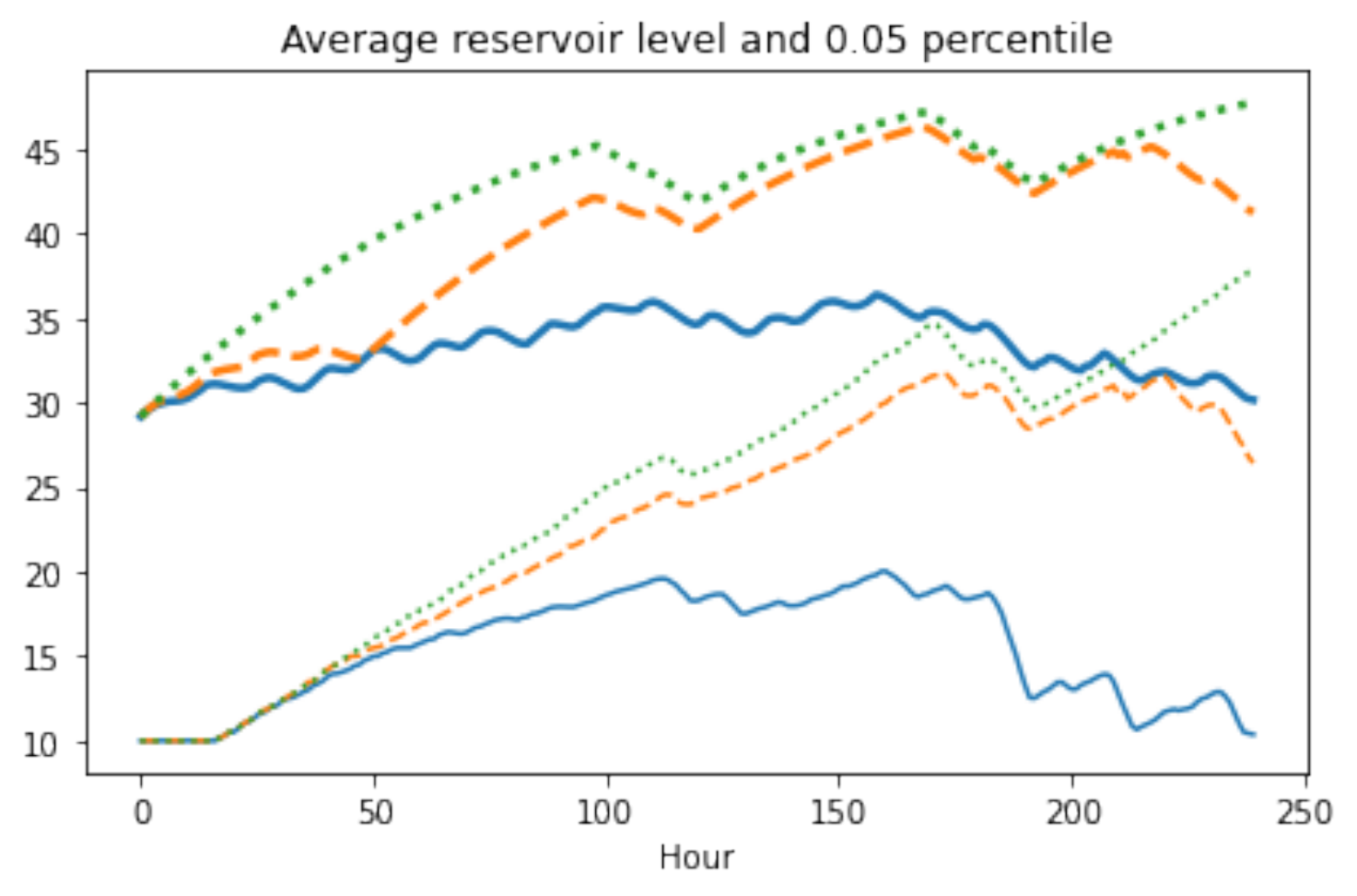}
	\caption{Plots for risk sensitivity $\theta=0, 0.01$ and $0.02$ of the reservoir level $M_t$ by hour: mean (thick blue solid, orange dashed and green dotted respectively) and 0.05 percentile (thinner).
	}
	\label{fig:M_t-samples}
\end{figure}

\begin{figure}[h!]
	\centering
	\includegraphics[width=0.75\textwidth]{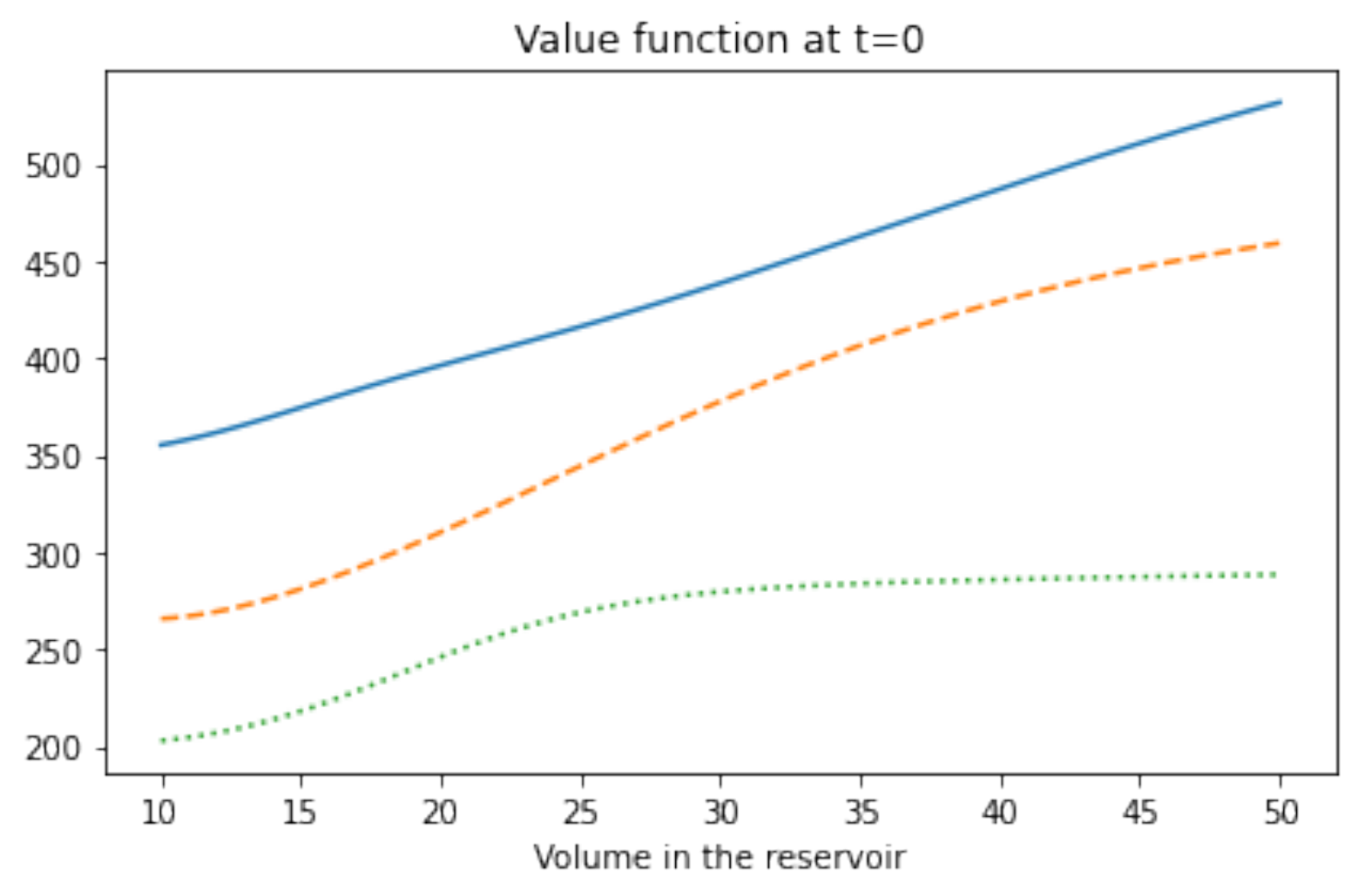}
	\caption{Predictions made by the $t=0$ value neural network for $\theta=0, 0.01$ and $0.02$ (blue solid, orange dashed and green dotted respectively) for the input $x_0=\left[m, 0, \mathbf{0}, \mathbf{0}\right]^{\transp}$, for different initial reservoir levels $m$.}
	\label{fig:VFs}
\end{figure}

\section*{Acknowledgments}
\noindent The authors would like to thank everyone whose suggestions helped improve the presentation of the paper. This work was partially supported by EPSRC grant numbers EP/N013492/1 and EP/P002625/1, by the Lloyd's Register Foundation-Alan Turing Institute programme on Data-Centric Engineering under the LRF grant G0095, and by the Swedish Energy Agency through grants number 42982-1 and 48405-1.

\section*{Data}
\noindent The code used in Section \ref{sec:num-ex-hydro-planning} can be found at \url{https://github.com/moriartyjm/optimalswitching/tree/main/hydro}.

\appendix

\section{Properties of conditional risk mappings}

Here we review definitions and preliminary results on conditional risk mappings that are used in the main text. References for this material include \cite{Frittelli2002,Detlefsen2005, PflugPichler2014,SDR2014,Ruszczynski2006,CHERIDITO2011,Ruszczynski2010,Follmer2016} among many others. Proofs are provided for results if they are not readily available in these references.

We are given a probability space $(\Omega,\mathcal{F},\mathbb{P})$ and a filtration $\mathbb{G} = \{\mathcal{G}_{t}\}_{t \in \mathbb{T}}$ of sub-$\sigma$-algebras of $\mathcal{F}$. All random variables below are defined with respect to this probability space, and (in-) equalities between random variables are in the $\PP$-almost-sure sense.

\subsection{Conditional risk mappings}\label{sec:conditional-risk-mapping}

A $\mathbb{G}$-conditional risk mapping is a family of mappings $\{\rho_{t}\}_{t \in \mathbb{T}}$, $\rho_{t} \colon L^{\infty}_{\mathcal{F}} \to L^{\infty}_{\mathcal{G}_{t}}$, satisfying for all $t \in \mathbb{T}$:
\begin{description}
	\item[\it Normalisation:] $\rho_{t}(0) = 0$,
	\item[\it Conditional translation invariance:] $\forall\; W \in L^{\infty}_{\mathcal{F}}$ and $Z \in L^{\infty}_{\mathcal{G}_{t}}$,
	\[
	\rho_{t}(Z + W) = Z + \rho_{t}(W),
	\]
	\item[\it Monotonicity:] $\forall\; W,Z \in L^{\infty}_{\mathcal{F}}$,
	\[
	W \le Z
	\implies \rho_{t}(W) \le \rho_{t}(Z).
	\]
\end{description}

For each $t \in \mathbb{T}$ we refer to $\rho_{t}$ as a conditional risk mapping. Note that in contrast to the one-step conditional risk mappings $\rho_t$ of \cite{Ruszczynski2010}, whose respective domains would be $L^{\infty}_{\mathcal{G}_{t+1}}$ in this context, here the domain of each $\rho_t$ is $L^{\infty}_{\mathcal{F}}$. Conditional risk mappings and the {\it monetary conditional risk mappings} of \cite{Follmer2016} are interchangeable via the mapping $Z \mapsto \rho_{t}(-Z)$. Each $\mathbb{G}$-conditional risk mapping satisfies the following property (cf. \cite[Proposition 3.3]{Cheridito2006}, \cite[Exercise 11.1.2]{Follmer2016}):
\begin{description}
	\item[\it Conditional locality:] for every $W$ and $Z$ in $L^{\infty}_{\mathcal{F}}$, $t \in \mathbb{T}$ and $A \in \mathcal{G}_{t}$,
	\[
	\rho_{t}(\mathbf{1}_{A}W + \mathbf{1}_{A^{c}}Z) = \mathbf{1}_{A}\rho_{t}(W) + \mathbf{1}_{A^{c}}\rho_{t}(Z).
	\]
\end{description}
A $\mathbb{G}$-conditional risk mapping is said to be strongly sensitive if it satisfies:
\begin{description}
	\item[\it Strong Sensitivity:] $\forall\; W,Z \in L^{\infty}_{\mathcal{F}}$ and $t \in \mathbb{T}$,
	\[
	W \le Z \;\;\text{and}\;\; \rho_{t}(W) = \rho_{t}(Z)
	\iff W = Z.
	\]
\end{description}
The strong sensitivity and monotonicity properties are sometimes jointly called the strict (or strong) monotonicity property.

\subsection{Aggregated conditional risk mappings}

\subsubsection{Finite horizon}
Where it simplifies notation we will write $W_{s:t} = (W_{s},\ldots,W_{t})$ for tuples of length $t-s+1$, with $W_{s:s} = W_{s}$, and use the component-wise partial order $W_{s:t} \le W'_{s:t} \iff W_{r} \le W'_{r},\; r = s,\ldots,t$. If $\alpha$ and $\beta$ are real-valued random variables then we write $\alpha W_{s:t} + \beta Z_{s:t} = (\alpha W_{s} + \beta Z_{s},\ldots,\alpha W_{t} + \beta Z_{t})$.

\begin{lem}\label{Lemma:Basic-Properties-Aggregator}\mbox{}
	The aggregated risk mapping $\{\rho_{s,t}\}$ has the following properties: for all $s,t \in \mathbb{T}$ with $s \le t$,
	\begin{description}
		\item[\it Normalisation:] $\rho_{s,t}(0,\ldots,0) = 0$.
		\item[\it Conditional translation invariance:] $\forall\; \{W_{r}\}_{r  = s}^{t} \in \otimes^{t-s+1} L^{\infty}_{\mathcal{F}}$ with $W_{s} \in \mathcal{G}_{s}$,
		\[
		\rho_{s,t}(W_{s},\ldots,W_{t}) = W_{s} + \rho_{s,t}(0,W_{s+1},\ldots,W_{t}).
		\]
		\item[\it Monotonicity:] $\forall\; \{W_{r}\}_{r  = s}^{t}, \{Z_{r}\}_{r  = s}^{t} \in \otimes^{t-s+1} L^{\infty}_{\mathcal{F}}$,
		\[
		W_{s:t} \le Z_{s:t}
		\implies \rho_{s,t}(W_{s:t}) \le \rho_{s,t}(Z_{s:t}).
		\]
		\item[\it Conditional locality:] $\forall$ $\{W_{r}\}_{r  = s}^{t}$ and $\{Z_{r}\}_{r  = s}^{t}$ in $\otimes^{t-s+1} L^{\infty}_{\mathcal{F}}$,
		\[
		\rho_{s,t}(\mathbf{1}_{A}W_{s:t}+\mathbf{1}_{A^{c}}Z_{s:t}) = \mathbf{1}_{A}\rho_{s,t}(W_{s:t}) + \mathbf{1}_{A^{c}}\rho_{s,t}(Z_{s:t}),
		\;\; \forall\,A \in \mathcal{G}_{s}.
		\]
		\item[\it Recursivity:] for each $s,r,t \in \mathbb{T}$ with $0 \le s < r \le t$,
		\[
		\rho_{s,t}(W_{s:t}) = \rho_{s,r}(W_{s:r-1},\rho_{r,t}(W_{r:t})).
		\]
	\end{description}
\end{lem}
\begin{proof}
	The proof follows by expanding the recursive definition of $\rho_{s,t}$ and using the properties of its generator.
\end{proof}

\subsubsection{Infinite horizon}
\begin{lem}\label{lem:vrho-limit}
	Recalling Definition \ref{def:Useful-Sequences-Random-Variables}, for all $W\in H_\mcF$ we have
	\[
	\varrho_{s}(W_{s},W_{s+1},\ldots) = \lim_{t \to \infty}\rho_{s,t}(W_{s},\ldots,W_{t}) \quad \forall s \in \mathbb{T}.
	\]
\end{lem}
\begin{proof}
	Let $W\in H_\mcF$ and $\{k_{t}\}_{t \in \mathbb{T}}$ be as in the definition of $H_\mathcal{F}$. Set $K_{t} \coloneqq \sum_{n \ge 0}k_{t+n}$. Note that $\{K_{t}\}_{t \in \mathbb{T}}$ is a non-negative, non-increasing deterministic sequence such that $\lim_{t \to \infty}K_{t} = 0$. For every $0 \le s \le t$ and $n \ge 1$,
	\begin{align*}
		\rho_{s,t+n}(W_{s},\ldots,W_{t+n}) & =
		\rho_{s,t+1}\big(W_{s},\ldots,W_{t},\rho_{t+1,t+n}(W_{t+1},\ldots,W_{t+n})\big) \\
		& \le \rho_{s,t+1}\left(W_{s},\ldots,W_{t},\sum_{m=1}^{n}k_{t+m}\right) \\
		& = \rho_{s,t}(W_{s},\ldots,W_{t}) + \sum_{m=1}^{n}k_{t+m}.
	\end{align*}
	Similarly we have
	\[
	\rho_{s,t+n}(W_{s},\ldots,W_{t+n}) \ge \rho_{s,t}(W_{s},\ldots,W_{t}) - \sum_{m=1}^{n}k_{t+m},
	\]
	and we conclude that $\PP$-almost surely, the sequence $\{\rho_{s,t}(W_{s},\ldots,W_{t})\}_{t \in \mathbb{T}}$ is Cauchy.
	
\end{proof}

\begin{lem}\label{lem:vrho-rec}
	For all $W\in H_\mcF$ we have
	\[
	\varrho_{s}(W_{s},W_{s+1},\ldots) = \rho_{s,s+1}(W_{s}, \varrho_{s+1}(W_{s+1},W_{s+2},\ldots)).
	\]
\end{lem}
\begin{proof}
	Arguing as in the proof of Lemma \ref{lem:vrho-limit}, there is a deterministic positive sequence $\{K_t\}_{t\in\bbT}$, with $\lim_{t \to \infty} K_t = 0$, such that for every $0 \le s \le t$ we have
	\begin{align*}
		|\varrho_{s+1}(W_{s+1},W_{s+2},\ldots)-\rho_{s+1,t}(W_{s+1},\ldots,W_t)|\leq K_t \quad \text{a.s.}
	\end{align*}
	The monotonicity and conditional translation invariance of $\rho_{s+1,t}$ imply that
	\begin{align*}
		\rho_{s,s+1}(W_{s}, \varrho_{s+1}(W_{s+1},W_{s+2},\ldots))&\leq\rho_{s,s+1}(W_{s}, \rho_{s+1,t}(W_{s+1},\ldots,W_t)+K_t)
		\\
		&=\rho_{s,t}(W_{s},W_{s+1},\ldots,W_t)+K_t.
	\end{align*}
	Taking the limit as $t\to\infty$ we find that
	\begin{align*}
		\rho_{s,s+1}(W_{s}, \varrho_{s+1}(W_{s+1},W_{s+2},\ldots))\leq \varrho_{s}(W_{s}, W_{s+1},\ldots).
	\end{align*}
	A similar argument can be applied to find the reverse inequality.
\end{proof}

All of the properties in Lemma \ref{Lemma:Basic-Properties-Aggregator} for finite sequences extend to infinite sequences in $H_\mathcal{F}$ with $\varrho_{s}$ playing the role of $\rho_{s,\infty}$.

\subsection{Martingales for aggregated conditional risk mappings}\label{sec:Aggregated-Martingales}

We close by presenting elementary martingale theory for aggregated conditional risk mappings (see also \cite{Follmer2016,Kratschmer2010}).

Let $f = \{f_{t}\}_{t \in \mathbb{T}}$ be a sequence in $L^{\infty}_{\mathcal{F}}$. We say that $W \in \mcL_\mathbb{G}^{\infty}$ is an $f$-extended $\{\rho_{s,t}\}$-sub (-super) martingale if:
\[
W_{s} \le (\ge)\, \rho_{s,t}\big(f_{s},\ldots,f_{t-1},W_{t}\big), \quad 0 \le s \le t,
\]
and an $f$-extended $\{\rho_{s,t}\}$ martingale if it has both these properties. Note that we use the convention
\[
\rho_{s,t}\big(f_{s},\ldots,f_{t-1},W_{t}\big) = \rho_{t,t}(W_{t}) \;\; \text{if}\;\; s = t.
\]
If $f \equiv 0$ then the qualifier ``$f$-extended'' is omitted.

\begin{lem}\label{lem:One-Step-Submartingale-Property}
	The definitive property for an $f$-extended $\{\rho_{s,t}\}$-sub (-super) martingale $W$
	is equivalent to the one-step property,
	\[
	W_{t} \le (\ge)\, \rho_{t,t+1}(f_{t},W_{t+1}),\;\; t \in \mathbb{T}.
	\]
\end{lem}
\begin{proof}
	If $\{W_{t}\}_{t \in \mathbb{T}}$ is a one-step $f$-extended $\{\rho_{s,t}\}$-submartingale then for all $s,t \in \mathbb{T}$ such that $s < t$ we have
	\begin{align*}
		\rho_{s,t}\big(f_{s},\ldots,f_{t-1},W_{t}\big) & = \rho_{s,t-1}\big(f_{s},\ldots,f_{t-2},\rho_{t-1,t}(f_{t-1},W_{t})\big) \\
		& \ge \rho_{s,t-1}\big(f_{s},\ldots,f_{t-2},W_{t-1}\big) \ldots \ge W_{s}.
	\end{align*}
	
	The case $s = t$ and the converse implication that an $f$-extended $\{\rho_{s,t}\}$-submartingale satisfies the one-step property are both trivial and thus omitted.
\end{proof}

\begin{lem}[Doob Decomposition]\label{Lemma:Doob-Decomposition}
	Let $W \in \mcL_\mathbb{G}^{\infty}$. There exists an almost surely unique $\{\rho_{s,t}\}$-martingale $M$ and $\mathbb{G}$-predictable process $A$ such that $M_{0} = A_{0}$ and
	\begin{equation}\label{eq:Doob-Decomposition}
		W_{t} = W_{0} + M_{t} + A_{t}.
	\end{equation}
	The processes $A$ and $M$ are defined recursively as follows:
	\begin{align*}
		& \begin{cases}
			A_{0} = 0,\\
			A_{t+1} = A_{t} + \left(\rho_{t}(W_{t+1}) - W_{t}\right), \;\; t \in \mathbb{T},
		\end{cases} \\
		& \begin{cases}
			M_{0} = 0,\\
			M_{t+1} = M_{t} + \left(W_{t+1} - \rho_{t}(W_{t+1})\right), \;\; t \in \mathbb{T}.
		\end{cases}
	\end{align*}
	If $W$ is a $\{\rho_{s,t}\}$-sub (-super) martingale then $A$ is increasing (decreasing).
\end{lem}
\begin{proof}
	Proved in the same way as Lemma 5.1 of \cite{Kratschmer2010}.
\end{proof}
\subsubsection{Optional stopping properties.}
First let $\tau \in \mathscr{T}$ be a stopping time. For sequences $\{f_t\}_{t \in \mathbb{T}}$ and $\{W_t\}_{t \in \mathbb{T}}$ in $H_\mcF$ define the aggregated cost $\rho_{t,\tau}(f_t,\ldots,f_{\tau-1},W_{\tau})$ as
\begin{equation}\label{eq:stopped-c-aggregated-cost}
	\rho_{t,\tau}(f_t,\ldots,f_{\tau-1},W_{\tau}) =
	\begin{cases}
		0,& \text{on} \enskip \{\tau < t\},\\
		\rho_{t}(W_{t}),& \text{on} \enskip \{\tau = t\},\\
		\rho_{t}\big(f_t+\rho_{t+1,\tau}(f_{t+1}),\ldots,f_{\tau-1},W_{\tau})\big),& \text{on} \enskip \{\tau > t\}.
	\end{cases}
\end{equation}

Given another stopping time $\varsigma \in \mathscr{T}$,
define the aggregated cost \\ $\rho_{\varsigma,\tau}(f_{\varsigma},\ldots,f_{\tau-1},W_{\tau})$ as
\begin{equation}\label{eq:doubly-stopped-aggregated-cost}
	\begin{split}
		\rho_{\varsigma,\tau}(f_{\varsigma},\ldots,f_{\tau-1},W_{\tau}) & = \sum_{t \in \mathbb{T}}\mathbf{1}_{\{\varsigma = t\}}\rho_{t,\tau}(f_{t},\ldots,f_{\tau-1},W_{\tau}) \\
		& = \begin{cases}
			0,& \text{on} \enskip \{\tau < \varsigma\},\\
			\rho_{\varsigma}(W_{\varsigma}),& \text{on} \enskip \{\tau = \varsigma\},\\
			\rho_{\varsigma}\big(f_{\varsigma}+\rho_{\varsigma+1,\tau}(f_{\varsigma+1},\ldots,f_{\tau-1},W_{\tau})\big),& \text{on} \enskip \{\tau > \varsigma\}.
		\end{cases}
	\end{split}
\end{equation}

Without loss of generality we can assume $\tau \ge t$ and $ \tau \ge \varsigma$ in \eqref{eq:stopped-c-aggregated-cost} and \eqref{eq:doubly-stopped-aggregated-cost} respectively. The following lemma shows that the recursive property of aggregated conditional risk mappings extends to stopping times.
\begin{lem}\label{lem:Recursive-Optional-Stopping}
	If $\varsigma$, $\tilde{\varsigma}$ and $\tau$ are \emph{bounded} stopping times in $\mathscr{T}$ such that $\varsigma \le \tilde{\varsigma} \le \tau$, then for all sequences $\{f_{t}\}_{t \in \mathbb{T}}$ and $\{W_{t}\}_{t \in \mathbb{T}}$ in $\mcL^\infty_{\mathcal{F}}$ we have
	\[
	\rho_{\varsigma,\tau}(f_{\varsigma},\ldots,f_{\tau-1},W_{\tau}) = \rho_{\varsigma,\tilde{\varsigma}}\big(f_{\varsigma},\ldots,f_{\tilde{\varsigma} - 1},\rho_{\tilde{\varsigma},\tau}(f_{\tilde{\varsigma}},\ldots,f_{\tau-1},W_{\tau})\big).
	\]
\end{lem}
\begin{proof}
	Since $\tau$ is bounded it follows that $\tau \in \mathscr{T}_{[0,T]}$ for some integer $0 < T < \infty$. Furthermore, by \eqref{eq:doubly-stopped-aggregated-cost} it suffices to prove for all $0 \le t \le T$ that
	\begin{align}\label{eq:sem1}
		\mathbf{1}_{\{\tilde{\varsigma} \ge t\}}\rho_{t,\tilde{\varsigma}}\big(f_{t},\ldots,f_{\tilde{\varsigma} - 1},\rho_{\tilde{\varsigma},\tau}(f_{\tilde{\varsigma}},\ldots,f_{\tau - 1},W_{\tau})\big) = \mathbf{1}_{\{\tilde{\varsigma} \ge t\}}\rho_{t,\tau}(f_{t},\ldots,f_{\tau-1},W_{\tau}).
	\end{align}
	By decomposing $\{\tilde{\varsigma} \ge t\}$ into the disjoint events $\{\tilde{\varsigma} = t\}$ and $\{\tilde{\varsigma} \ge t+1\}$
	we have
	\begin{align*}
		{} & \mathbf{1}_{\{\tilde{\varsigma} \ge t\}}\rho_{t,\tilde{\varsigma}}\big(f_{t},\ldots,f_{\tilde{\varsigma} - 1},\rho_{\tilde{\varsigma},\tau}(f_{\tilde{\varsigma}},\ldots,f_{\tau - 1},W_{\tau})\big) \\ = {} & \mathbf{1}_{\{\tilde{\varsigma} = t\}}\rho_{t,\tau}(f_{t},\ldots,f_{\tau-1},W_{\tau}) \\
		& + \mathbf{1}_{\{\tilde{\varsigma} \ge t+1\}}\rho_{t,t+1}\big(f_{t},\rho_{t+1,\tilde{\varsigma}}\big(f_{t+1},\ldots,f_{\tilde{\varsigma} - 1}, \rho_{\tilde{\varsigma},\tau}(f_{\tilde{\varsigma}},\ldots,f_{\tau - 1},W_{\tau})\big).
	\end{align*}
	If $t < T$ and if \eqref{eq:sem1} holds with $t+1$ in place of $t$ then using conditional translation invariance we get
	\begin{align*}
		{} & \mathbf{1}_{\{\tilde{\varsigma} \ge t\}}\rho_{t,\tilde{\varsigma}}\big(f_{t},\ldots,f_{\tilde{\varsigma} - 1},\rho_{\tilde{\varsigma},\tau}(f_{\tilde{\varsigma}},\ldots,f_{\tau - 1},W_{\tau})\big) \\ = {} & \mathbf{1}_{\{\tilde{\varsigma} = t\}}\rho_{t,\tau}(f_{t},\ldots,f_{\tau-1},W_{\tau}) \\
		& + \mathbf{1}_{\{\tilde{\varsigma} \ge t+1\}}\rho_{t,t+1}\big(f_{t},\rho_{t+1,\tilde{\varsigma}}\big(f_{t+1},\ldots,f_{\tilde{\varsigma} - 1},\rho_{\tilde{\varsigma},\tau}(f_{\tilde{\varsigma}},\ldots,f_{\tau - 1},W_{\tau})\big) \\
		= {} & \mathbf{1}_{\{\tilde{\varsigma} = t\}}\rho_{t,\tau}(f_{t},\ldots,f_{\tau-1},W_{\tau}) \\
		& + \mathbf{1}_{\{\tilde{\varsigma} \ge t+1\}}\rho_{t,t+1}\big(f_{t},\rho_{t+1,\tau}(f_{t+1},\ldots,f_{\tau-1},W_{\tau})\big) \\
		= {} & \mathbf{1}_{\{\tilde{\varsigma} \ge t\}}\rho_{t,\tau}(f_{t},\ldots,f_{\tau-1},W_{\tau}),
	\end{align*}
	and we conclude using backward induction.
\end{proof}


\end{document}